\documentclass[journal,twoside,web]{ieeecolor}
\usepackage{generic}
\usepackage{cite}
\usepackage{amsmath,amssymb,amsfonts}
\usepackage{algorithmic}
\usepackage{graphicx}
\usepackage{algorithm,algorithmic}
\usepackage{hyperref}
\usepackage{textcomp}
\usepackage{xcolor}

\usepackage{algorithm}
\usepackage{algorithmic} 
\usepackage{float}
\usepackage{float}
\usepackage{subfig}
\usepackage{caption}
\usepackage{tabularx}
\usepackage[flushleft]{threeparttable} 
\let\labelindent\relax
\usepackage{enumitem}
\usepackage{makecell} 
\usepackage{booktabs}

\usepackage{url}
\usepackage{color,hyperref}
\definecolor{darkblue}{rgb}{0,0,0.8}

\hypersetup{colorlinks,breaklinks,
	linkcolor=darkblue,urlcolor=darkblue,anchorcolor=darkblue,citecolor=darkblue}

\newtheorem{theorem}{Theorem}
\newtheorem{lemma}{Lemma}
\newtheorem{assumption}{Assumption}

\newtheorem{remark}{Remark}
\newtheorem{corollary}{Corollary}

\newcommand{\order}[1]{\mathcal{O}\left(#1\right)}
\newcommand{\prt}[1]{\left(#1\right)}
\newcommand{\brk}[1]{\left[#1\right]}
\newcommand{\crk}[1]{\left\{#1\right\}}

\newcommand{\orderi}[1]{\mathcal{O}(#1)}
\newcommand{\prti}[1]{(#1)}
\newcommand{\brki}[1]{[#1]}
\newcommand{\crki}[1]{\{#1\}}

\newcommand{\inpro}[1]{\left\langle #1 \right\rangle}
\newcommand{\condE}[2]{\E\brk{#1\middle|#2}}
\newcommand{\condEi}[2]{\E[#1|#2]}
\newcommand{\norm}[1]{\left\Vert #1 \right\Vert}
\newcommand{\normi}[1]{\Vert #1 \Vert}

\newcommand{\hLambda}{\hat{\Lambda}}

\newcommand{\R}{\mathbb{R}}
\newcommand{\E}{\mathbb{E}}

\newcommand{\x}{\mathbf{x}}
\newcommand{\z}{\mathbf{z}}

\newcommand{\1}{\mathbf{1}}

\newcommand{\T}{\intercal}
\newcommand{\sumn}{\sum_{i=1}^n}
\newcommand{\cN}{\mathcal{N}}
\renewcommand{\d}{\mathbf{d}}
\newcommand{\y}{\mathbf{y}}
\newcommand{\s}{\mathbf{s}}
\newcommand{\e}{\mathbf{e}}
\newcommand{\hU}{\hat{U}}
\newcommand{\diag}{\mathrm{diag}}

\newcommand{\fp}[2]{f_{#1, \pi_{#2}^{#1}}}

\newcommand{\bx}[2]{\bar{x}_{#1}^{#2}}

\newcommand{\xitl}{x_{i,t}^{\ell}}

\newcommand{\Fp}[1]{F_{\pi_{#1}}}
\newcommand{\Bxt}[1]{\1(\bx{t}{#1})^{\T}}

\newcommand{\cF}{\mathcal{F}}
\newcommand{\bzero}{\mathbf{0}}
\newcommand{\cL}{\mathcal{L}}
\newcommand{\fmn}{\frac{1}{mn}\sumn\sum_{\ell=1}^m f_{i,\ell}^*}
\newcommand{\summ}{\sum_{\ell = 1}^m}
\newcommand{\C}{\mathcal{C}}
\newcommand{\0}{\mathbf{0}}
\newcommand{\Ldrr}{\cL^{(\mathrm{D-RR})}}
\newcommand{\Qdrr}{Q^{(\mathrm{D-RR})}}
\newcommand{\sigfmn}{\sigma_f^*}
\newcommand{\cO}{\mathcal{O}}

\definecolor{cuhkpl}{RGB}{152,24,147}

\def\BibTeX{{\rm B\kern-.05em{\sc i\kern-.025em b}\kern-.08em
    T\kern-.1667em\lower.7ex\hbox{E}\kern-.125emX}}
\begin{document}
\title{Distributed Random Reshuffling Methods with Improved Convergence}
\author{Kun Huang, Linli Zhou, and Shi Pu, \IEEEmembership{Senior Member, IEEE}
\thanks{This work was supported by the National Natural Science Foundation of China under Grant 62003287. (Corresponding author: Shi Pu)}
\thanks{Kun Huang and Linli Zhou contribute equally. Kun Huang, Linli Zhou, and Shi Pu are with the School of Data Science, The Chinese University of Hong Kong,
Shenzhen 518172, China (e-mail: \{kunhuang, linlizhou\}@link.cuhk.edu.cn, pushi@cuhk.edu.cn).}}

\maketitle

\begin{abstract}
    This paper proposes two distributed random reshuffling methods, namely Gradient Tracking with Random Reshuffling (GT-RR) and Exact Diffusion with Random Reshuffling (ED-RR), to solve the distributed optimization problem over a connected network, where a set of agents aim to minimize the average of their local cost functions. Both algorithms invoke random reshuffling (RR) update for each agent, inherit favorable characteristics of RR for minimizing smooth nonconvex objective functions, and improve the performance of previous distributed random reshuffling methods both theoretically and empirically. Specifically, both GT-RR and ED-RR achieve the convergence rate of $\mathcal{O}(1/[(1-\lambda)^{1/3}m^{1/3}T^{2/3}])$ in driving the (minimum) expected squared norm of the gradient to zero, where $T$ denotes the number of epochs, $m$ is the sample size for each agent, and $(1-\lambda)$ represents the spectral gap of the mixing matrix. When the objective functions further satisfy the Polyak-{\L}ojasiewicz (PL) condition, we show GT-RR and ED-RR both achieve $\mathcal{O}(1/[(1-\lambda)mT^2])$ convergence rate in terms of the averaged expected differences between the agents' function values and the global minimum value. Notably, both results are comparable to the convergence rates of centralized RR methods (up to constant factors depending on the network topology) and outperform those of previous distributed random reshuffling algorithms. Moreover, we support the theoretical findings with a set of numerical experiments. 
\end{abstract}


\section{Introduction}
\label{sec:introduction}
    \IEEEPARstart{I}{n} this paper, we consider a group of networked agents, labelled $[n] := \crki{1, 2,\ldots, n}$, where each agent $i$ holds its own local cost function $f_i : \R^p \rightarrow \R$. 
    Specifically, each local cost function is associated with a local private dataset with $m$ data points or mini-batches. We investigate how this group of agents collaborate to solve the following optimization problem:
    \begin{equation}
        \label{eq:P_RR}
        \min_{x\in\R^p} \frac{1}{n}\sum_{i=1}^n f_i(x) \quad \text{{with}} \quad f_i(x)= \frac{1}{m}\sum_{\ell = 1}^m f_{i,\ell}(x).
    \end{equation}
    The finite sum structure of $f_i$ finds applications in various fields such as signal processing, distributed estimation, and machine learning. 
    In particular, within the context of distributed machine learning, dealing with the large volume of data, or large $nm$, can be challenging. One of the effective approaches is to consider decentralized stochastic gradient methods, which avoids the need to calculate the full gradients at every iteration and allows each agent to communicate only to its direct neighbors in a connected network. Such a decentralized implementation reduces the latency and improves the algorithmic robustness compared to centralized learning algorithms with a central server \cite{lu2021optimal,nedic2018network}. 

    Decentralized stochastic gradient methods have seen several algorithms being explored, including those presented in \cite{shi2015extra,pu2021distributed,pu2021sharp,tang2018d,yuan2020influence,lian2017can,xin2021improved,li2019decentralized}. The common approach is to use unbiased stochastic gradients conditioned on the previous information, which is typically achieved by uniformly sampling with replacement at each iteration. Recent works on decentralized stochastic gradient algorithms have focused on achieving comparable performance to centralized stochastic gradient descent (SGD) methods that also rely on unbiased stochastic gradients. However, in real-world scenarios, the algorithms used may involve sampling \emph{without} replacement. One common strategy is Random Reshuffling (RR) that is used in PyTorch \cite{paszke2019pytorch} and TensorFlow \cite{abadi2016tensorflow}. RR cycles through data points or mini-batches by permuting them uniformly at random at the beginning of each cycle (epoch). The data points or mini-batches are then selected sequentially following the permuted order for the gradient computation. Compared to SGD with uniform sampling strategy (unshuffled SGD), RR permits the utilization of all the data points in every epoch, leading to better theoretical and empirical performance \cite{bottou2009curiously,bottou2012stochastic,mishchenko2020random,nguyen2020unified,yun2022minibatch}.

    Despite the benefits of RR under the centralized setting, the study and development of decentralized RR methods is fairly limited and less advanced. Currently, only the convergence rate of distributed random reshuffling (D-RR) method, as shown in \cite{huang2023drr}, has demonstrated the superiority of RR compared to unshuffled SGD methods. However, decentralization can slow down the performance of RR, as evidenced by the convergence guarantee $\orderi{1/[(1-\lambda)^3T^{2/3}]}$ for smooth\footnote{Here, the term ``smooth'' refers to the Lipschitz continuity of the gradients of the objective functions.} nonconvex objective functions in \cite{huang2023drr}, where the term $(1-\lambda)$ denotes the spectral gap of the mixing matrix, and $T$ counts the epoch number. By comparison, the centralized RR method achieves a convergence rate of $\orderi{1/(m^{1/3}T^{2/3})}$. The spectral gap can become close to zero as the connectivity of the communication network deteriorates, resulting in a slower convergence rate for D-RR. The unsatisfactory convergence rate of D-RR can be traced back to the algorithmic structure of distributed gradient descent (DGD) method \cite{nedic2009distributed} as well as the conservative analysis in \cite{huang2023drr}.

    In light of the aforementioned issues of the state-of-the-art distributed RR methods, we propose in this work two novel decentralized RR algorithms, termed Gradient Tracking with Random Reshuffling (GT-RR), and Exact Diffusion with Random Reshuffling (ED-RR), to enhance the convergence guarantee and extend the scope of decentralized RR methods. 
    Both GT-RR and ED-RR impose additional algorithmic structures compared to D-RR, which precludes the direct adoption of the analysis for D-RR. Therefore, deriving improved theoretical guarantees for GT-RR and ED-RR requires innovative approaches as well as more involved analysis. Specifically, there are two immediate technical challenges to be addressed. Considering a random permutation $\pi^i := \crki{\pi^i_0, \pi^i_1, \ldots, \pi^i_{m-1}}$ of $\crki{1, 2, \ldots, m}$, the first challenge entails managing the difference of two successive stochastic gradients, $[\nabla \fp{i}{\ell + 1}(x_{i, t}^{\ell + 1}) - \nabla \fp{i}{\ell}(\xitl)]$, which arises from the structure of both proposed algorithms. Existing solutions rely heavily on the unbiasedness of stochastic gradients, which do not apply to RR type methods. The second challenge involves how to avoid the accumulation of gradient tracking errors under different permutations across various epochs when the gradient tracking technique is combined with RR. Both challenges arise from the unique properties of RR and contribute to the complexity of analyzing the proposed algorithms. To address them, the first idea is to consider a transformed and unified form of the algorithms, inspired by the work in \cite{alghunaim2021unified}, to avoid the two successive stochastic gradients and utilize the structure of RR updates. In addition, more involved analysis is performed, including considering epoch-wise errors and constructing new Lyapunov functions to build the key inequalities. More detailed discussions are deferred to Subsection \ref{subsec:form} and Section \ref{sec:preliminary}. 

    This paper shows that GT-RR and ED-RR inherit favorable characteristics of centralized RR method and deliver superior numerical and theoretical performance compared to the state-of-the-art decentralized methods. Specifically, our results indicate that GT-RR and ED-RR achieve the same convergence rate as centralized RR method (up to coefficients depending on the mixing matrix) for minimizing smooth nonconvex objective functions. These methods attain enhanced theoretical guarantees in terms of the sample size $m$ and the spectral gap $(1-\lambda)$ when compared to D-RR. In addition, we show that GT-RR and ED-RR enjoy improved convergence rate in terms of $(1-\lambda)$ under both constant and decreasing stepsizes compared to D-RR when the objective functions further satisfy the Polyak-{\L}ojasiewicz (PL) condition.\footnote{It is noteworthy that strongly convex functions satisfy the PL condition \cite{karimi2016linear}.}

    \subsection{Related Works}


    Distributed stochastic gradient methods are widely adopted to tackle large-scale machine learning problems where querying noisy or stochastic gradients is more practical than obtaining full gradients. The work presented in \cite{lian2017can} first demonstrated that distributed stochastic gradient descent (DSGD) can match the performance of centralized SGD method when minimizing smooth nonconvex objective functions. However, DSGD has been criticized for its inability to converge to the exact solution under a constant stepsize even when the full gradients are accessible. Several subsequent works, including gradient tracking based methods \cite{pu2021distributed,xin2021improved,lu2021optimal,koloskova2021improved,alghunaim2023enhanced,song2021optimal} and primal-dual like methods \cite{yuan2020influence,yuan2021removing,huang2021improving,tang2018d}, have proposed solutions to address such an issue. These methods are popular particularly because of the ability to relieve from data heterogeneity \cite{huang2023distributed,huang2022tackling}. Nevertheless, the works rely on the assumption of unbiased stochastic gradients and cannot be directly generalized to RR variants. 

    Random reshuffling (RR) strategy is commonly used in practice for solving large-scale machine learning problems despite the biased stochastic gradients \cite{ying2018stochastic}. Several works have established that RR outperforms unshuffled SGD, both theoretically and empirically, as demonstrated in \cite{gurbu2019,haochen2019,nguyen2020unified,mishchenko2020random,bottou2009curiously,bottou2012stochastic,lix2021convergence}. Specifically, RR achieves a convergence rate of $\orderi{1/(mT^2)}$ regarding the final iterate \cite{mishchenko2020random} when the objective function is smooth and strongly convex. Note that such a result is sharp in terms of $m$ and $T$ \cite{yun2022minibatch,cha2023tighter}. For smooth nonconvex objective functions, RR achieves an $\orderi{1/(m^{1/3}T^{2/3})}$ convergence rate concerning the minimum (or average) of the expected squared norm of the gradients \cite{mishchenko2020random,nguyen2020unified}. Both rates are superior to those of SGD, which behaves as $\orderi{1/(mnT)}$ and $\orderi{1/\sqrt{mnT}}$, respectively, for large $T$.


    Several recent works have extended RR to the decentralized setting \cite{yuan2018variance,jiang2022distributed,huang2023drr}. The work in \cite{yuan2018variance} combines Exact Diffusion \cite{yuan2018exact}, RR, and variance reduction to achieve linear convergence for smooth strongly convex objective functions. The paper \cite{jiang2022distributed} applies RR to a structured convex problem, and the proposed algorithm converges to a neighborhood of the optimal solution at a sublinear rate in expectation.
    The algorithm D-RR considered in \cite{huang2023drr}, which invokes the RR update for each agent, enjoys the superiority of RR over distributed unshuffled SGD methods under certain scenarios for both smooth nonconvex and strongly convex objective functions.

    It is worth noting that there exist other approaches that improve upon the theoretical guarantees of SGD for minimizing finite sum objective functions. For instance, several (distributed) variance reduction (VR)-based methods have been proposed to achieve exact linear convergence rate for minimizing smooth objective functions satisfying the PL condition or strong convexity, including \cite{reddi2016fast,xin2021fast,xin2021hybrid}.

\subsection{Main Contribution}

    The main contribution of this paper is {four}-fold. 
    
    Firstly, we propose two new algorithms, Gradient Tracking with Random Reshuffling (GT-RR, Algorithm \ref{alg:GT-RR}) and Exact Diffusion with Random Reshuffling (ED-RR, Algorithm \ref{alg:ED-RR}), to solve Problem \eqref{eq:P_RR} over networked agents. The proposed algorithms both inherit favorable characteristics of centralized RR (CRR, Algorithm \ref{alg:CRR}), which serves as the benchmark for comparison.
    Specifically, we demonstrate that both algorithms can drive the minimum expected squared norm of the gradient to zero, at a rate of $\orderi{1/[(1-\lambda)^{1/3}m^{1/3}T^{2/3}]}$ for minimizing general nonconvex objective functions, which is comparable to the results for centralized RR \cite{mishchenko2020random} in terms of the epoch counter $T$ and the sample size $m$. Furthermore, if the objective function satisfies the PL condition, the proposed algorithms can achieve a convergence rate of $\orderi{1/[m(1-\lambda)T^2]}$ using decreasing stepsizes in terms of the averaged expected function values to the global minimum. Alternatively, with a constant stepsize $\alpha$, the error decreases exponentially fast to a neighborhood of zero with a size of order $\orderi{m\alpha^2/(1-\lambda)}$. Both results are also comparable to that of centralized RR up to constants related to the mixing matrix. 
    Notably, GT-RR and ED-RR share similar procedures with DSGT and ED, respectively, but achieve enhanced convergence results when higher accuracy is demanded.

    Secondly, GT-RR and ED-RR enjoy better theoretical guarantees compared to the state-of-the-art method D-RR in \cite{huang2023drr}. Specifically, both methods attain enhanced theoretical guarantees in terms of the sample size $m$ and the spectral gap $(1-\lambda)$ for minimizing smooth nonconvex objective functions. In addition, GT-RR and ED-RR enjoy improved convergence rate in terms of the spectral gap $(1-\lambda)$ under both constant and decreasing stepsizes compared to D-RR when the objective functions further satisfy the Polyak-{\L}ojasiewicz (PL) condition. Comparisons between the performance of GT-RR, ED-RR and the other related methods under general nonconvex objective functions and functions satisfying the PL condition are presented in Table \ref{tab:summary_ncvx} and Table \ref{tab:summary_PL}, respectively.

    Compared to (distributed) variance reduction (VR)-based methods, GT-RR and ED-RR are easy to implement and do not require additional conditions such as large batch sizes in \cite{pan2020d}, fine-tuning of hyperparameters in \cite{xin2021hybrid}, large storage cost in \cite{reddi2016fast,xin2021improved}, and unbalanced computation of SVRG-type procedures in \cite{sun2020improving}.

    {Thirdly}, the analysis for the two proposed algorithms is based on a general framework. Such a framework can be applied to derive the convergence results for a wide range of distributed algorithms equipped with RR technique. Notably, such a framework simplifies the previous one in \cite{alghunaim2021unified} and comes with novel analysis that better fits the nature of RR updates.
    It is worth noting that the improvements enjoyed by the new algorithms not only come from the designed algorithmic structure but also result from the improved analysis.
    By using the refined intermediate results presented in this work, we are able to provide a better convergence result for D-RR when the objective function is smooth and nonconvex. Such a result demonstrates the non-trivial extension of the analysis techniques from the earlier work \cite{huang2023drr}. 
    
    Finally, we provide numerical experiments that corroborate the theoretical improvements of the new algorithms. In particular, we show that under different scenarios, GT-RR and ED-RR respectively achieve the best numerical performance compared to the other choices.

    \begin{table}[htbp]
        \centering 
        \setlength{\tabcolsep}{2pt}
        \begin{threeparttable}
        \begin{tabular}{@{}cc@{}}
        \toprule
        Algorithm                  & Convergence Rate                                                                                                                      \\ \midrule
        CRR & $\order{\frac{1}{\boldsymbol{m^{1/3}T^{2/3}}}}$ \cite{mishchenko2020random}\\ \midrule
        DSGD                       & $\order{\frac{1}{\sqrt{mnT}}+\frac{n^2}{(1-\lambda)^2mT}}$ \cite{lian2017can}                                                         \\
        D-RR                       & $\order{\frac{1}{(1-\lambda) T^{{2}/{3}}}}$  \cite{huang2023drr} \tnote{(a)}                                                             \\ \midrule
        \makecell[c]{D-RR \\ (\textbf{This work})}  & $\order{\frac{1}{\boldsymbol{(1-\lambda)^{2/3} m^{1/3} T^{2/3}} } + \frac{1}{(1-\lambda) T}}$                                         \\ \midrule
        DSGT                       & $\order{\frac{1}{\sqrt{mnT}}+\frac{n}{(1-\lambda)mT}+\frac{n}{(1-\lambda)^4m^2T^2}}$ \cite{alghunaim2021unified}                      \\
        \makecell[c]{GT-RR \\ (\textbf{This work})} & $\order{\frac{1}{\boldsymbol{(1-\lambda)^{1/3}m^{1/3} T^{2/3}}} + \frac{1}{(1-\lambda)T}+\frac{1}{m^{\frac{7}{3}}(1-\lambda)^{\frac{7}{3}}T^{\frac{5}{3}}} }$ 
        \\ \midrule
        ED/ $D^2$                  & $\order{\frac{1}{\sqrt{mnT}} + \frac{n}{(1-\lambda)mT} + \frac{n}{(1-\lambda)^2 m^2 T^2} }$ \cite{alghunaim2021unified}               \\
        \makecell[c]{ED-RR \\ (\textbf{This work})} & $\order{\frac{1}{\boldsymbol{(1-\lambda)^{1/3}m^{1/3} T^{2/3}}} + \frac{1}{(1-\lambda)T}+\frac{1}{m^{\frac{7}{3}}(1-\lambda)^{\frac{4}{3}}T^{\frac{5}{3}}} }$ 
        \\ \bottomrule
        \end{tabular}
        \begin{tablenotes}
            \item[(a)] The result is obtained by minimizing over the arbitrary constant $\eta$ in the original result $\orderi{1 / (\eta T^{2/3}) + \eta^2/[(1-\lambda)^3 T^{2/3}]}$ in \cite{huang2023drr}.
          \end{tablenotes}
        \end{threeparttable}
        \caption{A summary of related the theoretical results under smooth nonconvex objective functions using a constant stepsize.}
        \label{tab:summary_ncvx}
    \end{table}

    \begin{table}[htbp]
        \centering
        \setlength{\tabcolsep}{1pt}
        \begin{threeparttable}
        \begin{tabular}{@{}ccc@{}}
            \toprule
            Algorithm                                   & \makecell[c]{Final Error Bound \\ (Constant Stepsize $\alpha$)}                                                                                  & \makecell[c]{Convergence Rate\\ (Decreasing Stepsize)}                                                  \\ \midrule
            CRR & $\order{\boldsymbol{m\alpha^2}}$ \cite{mishchenko2020random,nguyen2020unified} & $\order{\frac{\log(T)}{\boldsymbol{m T^2}}}$\cite{nguyen2020unified} \\ \midrule
            DSGD                                        & $\order{\frac{\alpha}{n} + \frac{\alpha^2}{(1-\lambda)^2}}$ \cite{koloskova2020unified} \tnote{(b)}                    & \makecell[l]{$\order{\frac{1}{mnT} + \frac{1}{(1-\lambda)^2 m^2 T^2}}$\\ \cite{pu2021sharp} \tnote{(b)}}         \\
            D-RR \cite{huang2023drr}            & $\order{\frac{m\alpha^2}{(1-\lambda)^3} }$ \tnote{(b)}                                                                 & $\order{\frac{1}{(1-\lambda)^3 mT^2} }$\tnote{(b)}                                               \\ \midrule
            DSGT                                        & \makecell[l]{$\cO(\frac{\alpha}{n} + \frac{\alpha^2 }{1-\lambda}+ \frac{\alpha^4}{n(1-\lambda)^4})$\\ \cite{alghunaim2021unified}} & \makecell[l]{$\order{\frac{1}{mn T} +  \frac{1}{(1-\lambda)^3 m^2T^2} }$\\ \cite{pu2021distributed} \tnote{(b)}} \\
            \makecell[c]{GT-RR \\ (\textbf{This work})} & $\order{\frac{\boldsymbol{m\alpha^2}}{\boldsymbol{1-\lambda}} + \frac{m^4\alpha^4}{(1-\lambda)^2}}$                                & $\order{\frac{1}{\boldsymbol{(1-\lambda)m T^2}} + \frac{1}{(1-\lambda)^2 T^4} }$                  \\ \midrule
            ED/ $D^2$                                   & \makecell[l]{$\order{\frac{\alpha}{n} + \frac{\alpha^2}{(1-\lambda)} + \frac{\alpha^4}{n(1-\lambda)^3}}$\\ \cite{alghunaim2021unified}} & \makecell[l]{$\order{\frac{1}{mnT} + \frac{1}{(1-\lambda) m^2 T^2}}$\\ \cite{huang2021improving} \tnote{(b)}}    \\
            \makecell[c]{ED-RR \\ (\textbf{This work})} & $\order{\frac{\boldsymbol{m\alpha^2}}{\boldsymbol{1-\lambda}} + \frac{m^4\alpha^4}{(1-\lambda)^2}}$                                & $\order{\frac{1}{\boldsymbol{(1-\lambda)m T^2}} + \frac{1}{(1-\lambda)^2 T^4} }$                  \\ \bottomrule
        \end{tabular}
        \begin{tablenotes}
            \item[(b)] The results are obtained for smooth strongly convex objective functions.
          \end{tablenotes}
        \end{threeparttable}
        \caption{A summary of the related theoretical results under smooth objective functions satisfying the PL condition. 
        }
        \label{tab:summary_PL}
    \end{table}

\subsection{Notation and Assumptions}
\label{subsec:notations}	
	We consider column vectors throughout this paper unless specified otherwise. We use $x_{i, t}^{\ell} \in \R^p$ to describe the iterate of agent $i$ at the $t$-th epoch during the $\ell$-th inner loop. For the sake of clarity and presentation, we define stacked variables as follows:
    \begin{align*}
        &\x_t^\ell=
        \begin{pmatrix}  
          {x_{1,t}^{\ell}},{x_{2,t}^{\ell}},\ldots, {x_{n,t}^{\ell}}
        \end{pmatrix}^{\T} \in \R^{n\times p},\\
        &\nabla F(\x_t^\ell)=
        \begin{pmatrix}
          \nabla f_1(x_{1,t}^\ell), \nabla f_2(x_{2,t}^\ell),\ldots, \nabla f_n(x_{n,t}^{\ell})
        \end{pmatrix}^{\T} \in \R^{n\times p}, \\
        &\nabla\Fp{\ell}(\x_t^\ell)=
        \begin{pmatrix}  
          \nabla \fp{1}{\ell}(x_{1,t}^\ell),\ldots,
          \nabla \fp{n}{\ell}(x_{n,t}^\ell) 
        \end{pmatrix}^{\T} \in \R^{n\times p}.
    \end{align*}

    We use $\bar{x}\in\R^p$ to denote the averaged variables (among agents). For instance, the variables $\bar{x}_t^{\ell}:= 1/n\sumn x_{i, t}^{\ell}$ and $\bar{\nabla}\Fp{\ell}(\x_t^{\ell}):= 1/n\sumn \nabla \fp{i}{\ell}(\xitl)$ denote the average of all the agents' iterates and shuffled gradients at the $t-$th epoch during the $\ell-$th inner loop, respectively. 
	We use $\norm{\cdot}$ to denote the Frobenius norm for a matrix and the $\ell_2$ norm for a vector. The term $\inpro{a, b}$ stands for the inner product of two vectors $a, b\in\R^{p}$. For two matrices $A, B\in\R^{n\times p}$, $\inpro{A, B}$ is defined as $\inpro{A, B} := \sum_{i=1}^n\inpro{A_i, B_i}$, where $A_i$ (and $B_i$) represents the $i-$row of $A$ (and $B$).

    Regarding the underlying assumptions, first we consider the standard assumption in the distributed optimization literature on the communication network. Specifically, we assume the agents in the network are connected via a graph $\mathcal{G}$, where $\mathcal{G}=(\mathcal{N},\mathcal{E})$ with $\mathcal{N}=\crki{1,2,...,n}$ representing the set of agents and $\mathcal{E}\subseteq \mathcal{N} \times \mathcal{N}$ representing the set of edges connecting the agents. In particular, $(i,i)\in\mathcal{E}$ for all $i\in\mathcal{N}$. The set of neighbors for agent $i$ is denoted by $\mathcal{N}_i=\{j\in \mathcal{N}:(i,j)\in \mathcal{E}\}$.
    The element $w_{ij}$ in the weight matrix $W\in\mathbb{R}^{n\times n}$ represents the weight of the edge between agents $i$ and $j$.
        \begin{assumption}
            \label{as:graph}
            The graph $\mathcal{G}$ is undirected and strongly connected, i.e., there exists a path between any two nodes in $\mathcal{G}$. There is a direct link between $i$ and $j$ $(i\neq j)$ in $\mathcal{G}$ if and only if $w_{ij}>0$ and $w_{ji}>0$; otherwise, $w_{ij}=w_{ji}=0$. The mixing matrix is nonnegative, symmetric, and stochastic, i.e., $W=W^{\T}$ and $W\1=\1$.
        \end{assumption} 

        Assumption \ref{as:graph} guarantees the spectral norm $\lambda$ of the matrix $(W - \1\1^{\T}/n)$ is strictly less than one \cite[Lemma 1]{huang2023drr}. As a consequence, we measure the impact of the network topology through the spectral gap $(1-\lambda)$, where a smaller spectral gap generally indicates a worse connectivity of the corresponding graph. 

        Assumption \ref{as:comp_fun} specifies the requirements for the objective functions, which include smoothness and lower boundedness.
        \begin{assumption}
            \label{as:comp_fun}	
            Each $f_{i,\ell}:\R^p\to\R$ is $L$-smooth and bounded from below, i.e., for all $x,x^\prime \in \R^p$ and $i,\ell$, we have $\normi{\nabla f_{i,\ell}(x)- \nabla f_{i,\ell}(x^\prime)} \leq L\normi{x-x^\prime}$.
        \end{assumption}

        In addition to the general nonconvex assumption above, we will consider a specific nonconvex condition known as the Polyak-{\L}ojasiewicz (PL) condition in Assumption \ref{as:PL}. Overparameterized models often satisfy this condition. Notably, the strong convexity condition implies the PL condition \cite{karimi2016linear}.
        
        \begin{assumption}
            \label{as:PL}
            There exist $\mu>0$, such that the aggregate function $f(x)=\frac{1}{n}\sum_{i=1}^n f_i(x)$ satisfies $2\mu\prti{f(x)-f^*} \leq \normi{\nabla f(x)}^2$, 
            for all $x \in \R^p$, where $f^*:= \inf_{x\in\R^p}f(x)$.
        \end{assumption}
	
	\subsection{Organization}

        The rest of this paper is organized as follows. In Section \ref{sec:algorithm}, we introduce the standing assumptions and the two new algorithms. 
        We then proceed to conduct a preliminary analysis in Section \ref{sec:preliminary}, followed by the presentation of the main results in Section \ref{sec:main}. Numerical simulations are provided in Section \ref{sec:sims}, and we conclude the paper in Section \ref{sec:conclusions}. 

        \section{Two Distributed Random Reshuffling Algorithms}
        \label{sec:algorithm}

        In this section, we introduce two algorithms, Gradient Tracking with Random Reshuffling (GT-RR) and Exact Diffusion with Random Reshuffling (ED-RR), with the goal of minimizing the impact of the network topology on the convergence rate compared to the existing algorithms, while achieving comparable performance as centralized RR method.
        Note that the algorithmic structures of GT-RR and ED-RR prevent us from adopting the analysis in \cite{huang2023drr} and those for distributed unshuffled stochastic gradient methods, such as \cite{pu2021distributed,huang2021improving,yuan2020influence}. One of the challenges arises from analyzing consecutive shuffled gradients in the form of $[\nabla \fp{i}{\ell + 1}(x_{i, t}^{\ell + 1}) - \nabla \fp{i}{\ell}(\xitl)]$ given a random permutation $\pi^i := \crki{\pi^i_0, \pi^i_1, \ldots, \pi^i_{m-1}}$ of $\crki{1, 2, \ldots, m}$. To tackle this challenge, we adopt a unified framework inspired by recent works including \cite{xu2021distributed,alghunaim2021unified,yuan2021decentralized}
        and develop new analysis strategies that utilize the properties of RR updates. Specifically, the recursions of the whole epochs are considered, and new Lyappunov functions are constructed to build the key inequalities. Details are given in Section \ref{sec:preliminary}.

	\subsection{Algorithms}
	\label{subsec:alg}
        
        We introduce the two algorithms to solve Problem \eqref{eq:P_RR}, Gradient Tracking with Random Reshuffling (GT-RR) in Algorithm \ref{alg:GT-RR} and Exact Diffusion with Random Reshuffling (ED-RR) in Algorithm \ref{alg:ED-RR}.\footnote{Note that the algorithms can be generalized to the case where each agent has a different number of data points or mini-batches $m_i$. Each agent simply follows its own sampling cycle. After introducing $m':=\text{Least Common Multiple}(m_1,m_2,\ldots,m_n)$, we can replace $m$ with $m'$ in the analysis.} Both algorithms employ RR updates for each agent and can be considered as extensions of distributed stochastic gradient tracking (DSGT) \cite{pu2021distributed} and Exact Diffusion (ED)/$D^2$/Exact Diffusion with adaptive step sizes (EDAS) \cite{huang2021improving,yuan2020influence,tang2018d} to incorporate RR updates, respectively. 

        	\begin{algorithm}[htbp]
		\begin{algorithmic}[1]
			\STATE Initialize $x_{i,0}^0$ for each agent $i\in\mathcal{N}$, determine $W = [w_{ij}]\in\R^{n\times n}$, stepsize $\{\alpha_t\}$.
			\FOR{Epoch $t = 0, 1, 2,\ldots, T-1$}
			\FOR{Agent $i$ in parallel}
			\STATE Independently sample a random permutation $\crk{\pi_{0}^i, \pi_{1}^i,\ldots, \pi_{m-1}^i}$ of $\crk{1,2,\ldots, m}$.
			\STATE Set $y_{i, t}^0 = \nabla \fp{i}{0}(x_{i,t}^0)$.
			\FOR{$\ell = 0, 1,\ldots, m-1$}
			\STATE Update $x_{i,t}^{\ell + \frac{1}{2}} = x_{i, t}^{\ell} - \alpha_t y_{i, t}^{\ell}$ and send $x_{i,t}^{\ell + \frac{1}{2}}$ to its neighbors $j\in\mathcal{N}_i$.
			\STATE Receive $x_{j, t}^{\ell + \frac{1}{2}}$ from its neighbors $j\in\mathcal{N}_i$ and update $x_{i, t}^{\ell + 1} = \sum_{j\in\mathcal{N}_i}w_{ij}x_{j,t}^{\ell + \frac{1}{2}}$. \label{line:gt-rr_combine_x}
			\STATE Send $y_{i, t}^{\ell}$ to its neighbors $j\in\mathcal{N}_i$.
			\STATE Receive $y_{j, t}^{\ell}$ from its neighbors and update for $\ell\neq m-1.$ Update
			$
			y_{i, t}^{\ell + 1} = \sum_{j\in\mathcal{N}_i} w_{ij}y_{j, t}^{\ell} + \nabla \fp{i}{\ell + 1}(x_{i, t}^{\ell + 1}) - \nabla \fp{i}{\ell}(x_{i, t}^{\ell}).
			$\label{line:gt-rr_y}
			\ENDFOR
			\STATE $x_{i, t+ 1}^0 = x_{i, t}^m$.
			\ENDFOR
			\ENDFOR
			\STATE Output $x_{i, T}$.
		\end{algorithmic}
		\caption{Gradient Tracking with Random Reshuffling (GT-RR)}
		\label{alg:GT-RR}
	\end{algorithm}

        

	\begin{algorithm}[htbp]
		\begin{algorithmic}[1]
			\STATE Initialize $x_{i,0}^0$ for each agent $i\in\mathcal{N}$, determine $W = [w_{ij}]\in\R^{n\times n}$, stepsize $\{\alpha_t\}$.
			\FOR{Epoch $t = 0, 1, 2,\ldots, T-1$}
			\FOR{Agent $i$ in parallel}
			\STATE Independently sample a random permutation $\crk{\pi_{0}^i, \pi_{1}^i,\ldots, \pi_{m-1}^i}$ of $\crk{1,2,\ldots, m}$.
			\FOR{$\ell = 0, 1,\ldots, m-1$}
			\IF{$\ell = 0$}
			\STATE Update $x_{i,t}^{\ell + \frac{1}{2}} = x_{i,t}^\ell - \alpha_t \nabla \fp{i}{\ell}(x_{i,t}^\ell)$.
			\ELSE 
			\STATE Update $x_{i,t}^{\ell + \frac{1}{2}} = 2x_{i,t}^\ell - x_{i,t}^{\ell - 1} - \alpha_t(\nabla \fp{i}{\ell}(\xitl) - \nabla \fp{i}{\ell-1}(x_{i,t}^{\ell - 1}))$. \label{line:ed-rr_x}
			\ENDIF
			\STATE Send $x_{i,t}^{\ell + \frac{1}{2}}$ to its neighbors $j\in\cN_i$. 
			\STATE Receive $x_{j,t}^{\ell + \frac{1}{2}}$ from its neighbors $j\in\cN_i$ and update $x_{i,t}^{\ell + 1} = \sum_{j\in\cN_i}w_{ij}x_{j,t}^{\ell + \frac{1}{2}}$.
			\ENDFOR
			\STATE $x_{i, t+ 1}^0 = x_{i, t}^m$.
			\ENDFOR
			\ENDFOR
			\STATE Output $x_{i, T}$.
		\end{algorithmic}
		\caption{Exact Diffusion with Random Reshuffling (ED-RR)}
		\label{alg:ED-RR}
	\end{algorithm}
 
 For both algorithms, each agent $i$ first generates a random permutation $\crki{\pi_0^i, \pi_1^i,\ldots, \pi_{m-1}^i}$ of $[m]$ and then performs $m$ distributed gradient steps correspondingly.
        In contrast to the unshuffled distributed stochastic gradient methods, each agent in GT-RR and ED-RR has guaranteed access to its full local data in every epoch. Specifically, in GT-RR, agent $i$ first performs an approximate stochastic gradient descent step with a gradient tracker $y_{i,t}^\ell$ initialized at $y_{i,t}^0 = \nabla \fp{i}{0}(x_{i,t}^0)$. Then the intermediate result is exchanged with neighbors and combined in line \ref{line:gt-rr_combine_x} of Algorithm \ref{alg:GT-RR}. The gradient tracker $y_{i,t}^\ell$ is also mixed and updated as in Line \ref{line:gt-rr_y}. In ED-RR, agent $i$ performs similar procedures but with a different local update scheme in Line \ref{line:ed-rr_x} of Algorithm \ref{alg:ED-RR}. Such a step can be viewed as a combination of a local gradient step $\psi_{i,t}^{\ell + 1} = x_{i,t}^\ell - \alpha_t \nabla \fp{i}{\ell}(\xitl)$ and a correction step $x_{i,t}^{\ell + \frac{1}{2}} = x_{i,t}^\ell + \psi_{i,t}^{\ell + 1} - \psi_{i,t}^{\ell}$. 
        ED-RR saves one communication step per inner update compared to GT-RR. It is worth noting that in contrast to the approach presented in \cite{yuan2018variance}, ED-RR achieves the convergence rate of $\orderi{1/(mT^2)}$ for minimizing smooth objective functions satisfying the PL condition without necessitating additional variance reduction procedures.

        When compared with centralized RR methods, the rationale behind the effectiveness of GT-RR and ED-RR is similar to that of D-RR \cite{huang2023drr}. These algorithms can be regarded as approximate implementations of the centralized RR method for solving an equivalent form of Problem \eqref{eq:P_RR} considering relation \eqref{eq:avg} below
        \begin{align}
            \label{eq:avg}
            \bx{t}{\ell + 1} = \bx{t}{\ell} - \frac{\alpha_t}{n}\sumn\nabla \fp{i}{\ell}(\xitl).
        \end{align}
        In GT-RR, with initialization $y_{i,t}^0 = \nabla \fp{i}{0}(x_{i,t}^0)$ at the beginning of each epoch, the so-called gradient tracking property is maintained, i.e., $\bar{y}_t^{\ell} = 1/n\sumn \nabla \fp{i}{\ell}(x_{i,t}^{\ell})$. This enables GT-RR to satisfy relation \eqref{eq:avg}.

        To see why relation \eqref{eq:avg} holds for ED-RR, note that ED-RR is equivalent to the following primal-dual like update by considering the matrix $(I-W)^{1/2}$ and introducing auxiliary variables $d_{i,t}^{\ell}$ for all $i\in[n]$ (similar to the arguments provided in \cite{huang2021improving,yuan2018exact}): 
        \begin{subequations}
            \label{eq:ED-RR}
            \begin{align}
                \x_t^{\ell + 1} &= W(\x_t^\ell - \alpha_t \nabla \Fp{\ell}(\x_t^\ell)) - (I-W)^{\frac{1}{2}} \d_t^{\ell}\\
                \d_t^{\ell + 1} &= \d_t^{\ell} + (I-W)^{\frac{1}{2}} \x_t^{\ell + 1},
            \end{align}
        \end{subequations}
        where we initialize $d_{i,0}^0 =\0$, $\forall i$. 
        Such an equivalent update readily implies relation \eqref{eq:avg} for ED-RR.


    
    In the remainder of this section, we present a unified framework for studying the convergence properties of both GT-RR and ED-RR.

    \subsection{A Unified Form for Analysis}
	\label{subsec:form}
    
    The update rules used by GT-RR and ED-RR in each epoch are similar to those of some unshuffled decentralized stochastic gradient methods. Such an observation motivates us to utilize a unified framework inspired by those proposed for distributed gradient methods \cite{yuan2021decentralized,alghunaim2021unified,xu2021distributed}. The unified framework helps avoid handling consecutive shuffled gradients, as outlined in \eqref{eq:abc_epoch}, and enables us to conduct a joint analysis for GT-RR and ED-RR.
    
    Specifically, we unify GT-RR and ED-RR by first introducing the matrices $A$, $B$, and $C\in\R^{n\times n}$ satisfying Assumption \ref{as:abc}, which is a mild condition that holds true for many distributed gradient methods. Similar assumptions appear in \cite{alghunaim2021unified}.
    \begin{assumption}
        \label{as:abc}
        The matrices $A, B^2,C\in\R^{n\times n}$ are chosen as a polynomial function of $W$: $ A=\sum_{d=0}^{p}a_d W^d$, $B^2=\sum_{d=0}^{p}b_d W^d$, $C=\sum_{d=0}^{p}c_d W^d$, 
            where $p\geq 0$ is an integer. Moreover, the constants $\{a_d,b_d,c_d\}_{d=0}^p$ are chosen such that $A$ and $C$ are doubly stochastic and the matrix $B$ satisfies $B\x=\0$ if and only if $x_1=x_2=\ldots=x_n$.
    \end{assumption}
    
    Based on Assumption \ref{as:abc}, we consider update \eqref{eq:abc_epoch} below which can be regarded as a unified algorithmic framework for distributed random reshuffling methods over networks allowing for different choices of matrices $A$, $B$, $C$,
    \begin{equation}
        \label{eq:abc_epoch}
        \begin{aligned}
            \x_t^{\ell + 1} &= A(C\x_t^\ell - \alpha_t \nabla \Fp{\ell}(\x_t^\ell)) - B\z_t^\ell,\\
            \z_t^{\ell + 1} &= \z_t^\ell + B\x_t^{\ell + 1},\quad \ell = 0,1,\ldots, m-1,\\
            \x_{t+1}^0 &= \x_{t}^m,
        \end{aligned}
    \end{equation}
     with proper initialization $\z_t^0= h(\x_t^0)$ ($t=0,1,\ldots$) for some mapping $h:\R^{n\times p}\rightarrow \R^{n\times p}$.
In particular, the updates in the $t$-th epoch of GT-RR and ED-RR can both be written as special cases of \eqref{eq:abc_epoch}: letting $A = W$, $B = I-W$, and $C=W$ in \eqref{eq:abc_epoch} recovers the update of GT-RR; similarly, choosing $A = W$, $B = (I-W)^{1/2}$, and $C = I$ recovers the update of ED-RR. Such argument is formalized in Lemma \ref{lem:abc}, which introduces a transformed form of \eqref{eq:abc_epoch} for the ease of the follow-up analysis. Similar ideas appear in \cite{alghunaim2021unified,huang2021improving}.


    \begin{lemma}
        \label{lem:abc}
        Let Assumptions \ref{as:graph} and \ref{as:abc} hold. Then both GT-RR and ED-RR can be written as 
        \begin{equation}
            \label{eq:abc}
            \begin{aligned}
                \x_t^{\ell + 1} &= (AC - B^2)\x_t^\ell - \alpha_t A(\nabla \Fp{\ell}(\x_t^\ell) - \nabla F(\Bxt{0})) - \s_t^\ell,\\
                \s_t^{\ell + 1} &= \s_t^\ell + B^2\x_t^\ell,\quad \ell = 0,1,\ldots, m-1,
            \end{aligned}
        \end{equation}
        where $\s_t^\ell := B(\z_t^\ell - B\x_t^\ell) + \alpha_t A\nabla F(\Bxt{0})$ for epoch $t = 0, 1, 2,\ldots, T-1$.
        In particular, letting $A = W$, $B = I-W$, $C= W$, and $\z_t^0 = -W\x_t^0$
        recovers the update of GT-RR. 
        Letting $A = W$, $B = (I-W)^{1/2}$, $C = I$, and $\z_t^0 = \0$
        recovers the update of ED-RR.
    \end{lemma}

    \begin{proof}
        See Appendix \ref{app:lem_abc}.
    \end{proof}


    \begin{remark}
        The main difference between the unified form \eqref{eq:abc} and \cite[(14)]{alghunaim2021unified} lies in 
        the new definition of $\s_t^\ell$ in \eqref{eq:abc}.  Such a difference is critical as it avoids dealing with an extra error term corresponding to $\alpha^4\normi{\bar{g}_k}^2$ in \cite{alghunaim2021unified} and simplifies the analysis. Here, the term $\bar{g}_k$ represents the averaged stochastic gradients in \cite{alghunaim2021unified}. More details can be found in Lemma \ref{lem:re_epoch} and Lemma \ref{lem:cons_e} in the next section. As a result, we are able to obtain the same theoretical guarantees for GT-RR and ED-RR when the objective functions satisfy the PL condition as illustrated in Corollary \ref{cor:PL_DGT_RR} and Corollary \ref{cor:PL_ED_RR}. By contrast, the results in \cite{alghunaim2021unified} indicate different convergence rates for DSGT and ED. The gap is only eliminated later in \cite{alghunaim2023enhanced,ye2022snap} for strongly convex objective functions.
    \end{remark}    

	\begin{remark}
		Although the D-RR algorithm proposed in \cite{huang2023drr} can be recovered using \eqref{eq:abc} by setting $A=W$, $C=I$, $B=\bzero$, and $\s_t^\ell=\alpha_t W \nabla F(\Bxt{0})$, this choice violates the assumption stated in Assumption \ref{as:abc} which requires that $B\x = \0$ if and only if $x_1=x_2=\ldots=x_n$. Consequently, the subsequent analysis based on the framework \eqref{eq:abc} does not apply to D-RR.
	\end{remark}

    Next, we focus on the transformed recursion \eqref{eq:abc} in Lemma \ref{lem:abc} to conduct the analysis. The key idea is to leverage the eigenvalue decomposition of the symmetric and stochastic matrix $W\in\R^{n\times n}$. Define $\hLambda$ as $\hLambda:= \diag(\lambda_2,\ldots, \lambda_n)$, where $1=\lambda_1>\lambda_2\ge\lambda_3\ge\cdots\ge\lambda_n$ are the eigenvalues of $W$. Then,
    \begin{equation}
        \label{eq:W_decomp}
        W = \begin{pmatrix}
            \frac{1}{\sqrt{n}}\1 & \hU 
        \end{pmatrix}
        \begin{pmatrix}
            1 & 0\\
            0 & \hLambda
        \end{pmatrix}
        \begin{pmatrix}
            \frac{1}{\sqrt{n}}\1^{\T}\\
            \hU^{\T}
        \end{pmatrix},
    \end{equation}
    where $\hU\hU^{\T} = I - \frac{\1\1^{\T}}{n}, \ \hU^{\T}\hU = I_{n-1}.$

    Based on Assumption \ref{as:abc}, the matrices $A$, $B^2$, and $C$ have similar eigenvalue decomposition. This observation leads to Lemma \ref{lem:re} below. Note that a similar approach was adopted in \cite{alghunaim2021unified}.

    \begin{lemma}
		\label{lem:re}
		Let Assumptions \ref{as:graph}, \ref{as:comp_fun}, and \ref{as:abc} hold. 
        We have
		\begin{align*}
             \bar{x}_t^{\ell+1} &= \bar{x}_t^{\ell}-\alpha_t \bar{\nabla} \Fp{\ell}(\x_t^\ell), \\
			\e_t^{\ell + 1} &= \Gamma \e_t^{\ell} - \alpha_t V^{-1} \begin{pmatrix}
				\hLambda_a \hU^{\T}\prt{\nabla \Fp{\ell}(\x_t^{\ell}) - \nabla F(\Bxt{0})}\\
				\0
			\end{pmatrix},
		\end{align*}
		for $\ell = 0,1,\ldots, m-1$ and epoch $t = 0, 1, 2,\ldots, T-1$. Here $\Gamma\in\R^{2(n-1)\times 2(n-1)}$ is given by
        \begin{align*}
            \Gamma := V^{-1}\begin{pmatrix}
                \hLambda_a\hLambda_c - \hLambda_b^2 & - \hLambda_b \\
                \hLambda_b & I_{n-1} \\
            \end{pmatrix}V
        \end{align*}
        for some invertible matrix $V$ and satisfies $\gamma:= \normi{\Gamma} < 1$,
        where $\hat{\Lambda}_a$, $\hat{\Lambda}_b^2$ and $\hat{\Lambda}_c$ are diagonal matrices composed of eigenvalues of $A$, $B^2$, and $C$, respectively: $\hat{\Lambda}_a=\diag(\lambda_{a,1}, \ldots, \lambda_{a,n})$, $\hat{\Lambda}_b^2=\diag(\lambda_{b,1}^2, \ldots, \lambda_{b,n}^2)$, and $\hat{\Lambda}_c=\diag(\lambda_{c,1}, \ldots, \lambda_{c,n}).$

        In addition, $\e_t^\ell\in\R^{2(n-1)\times p}$ is defined as 
        \begin{align*}
                       \e_t^{\ell}:= V^{-1}
            \begin{pmatrix}
            \hU^{\T}\x_t^\ell\\
            \hLambda_b^{-1}\hU^{\T}\s_t^\ell
        \end{pmatrix}.
        \end{align*}
	\end{lemma}
        \begin{proof}
            See Appendix \ref{app:re}.
        \end{proof}

        \begin{remark}
            \label{rem:gamma}
            It can be verified that $\orderi{1/(1-\gamma^2)}\sim \orderi{1/(1-\lambda)}$, as detailed in \cite[Appendix B]{alghunaim2021unified}. 
        \end{remark}

        \begin{remark}
            \label{rem:cons2e}
            It is worth noting that since $\s_t^0 =B h(\x_t^0) -B^2\x_t^0 + \alpha_t A\nabla F(\Bxt{0})$, and $\x_{t+1}^0 = \x_{t}^m$, we have $\e_{t+1}^0 = \e_t^m$. The consensus error $\normi{\x_t^\ell-\1\bx{t}{\ell}}^2$ is upper bounded by $\normi{V}^2\normi{\e_t^{\ell}}^2$ according to \cite[Remark 3]{alghunaim2021unified}, that is, $\normi{\x_t^\ell-\1\bx{t}{\ell}}^2
            \leq \normi{V}^2\normi{\e_t^{\ell}}^2.$
        \end{remark}

\section{Preliminary Analysis}
\label{sec:preliminary}

    In this section, we present several preliminary results. The primary objective is to locate a Lyapunov function $Q_t$ in the form of \eqref{eq:lyapunov} to trigger the convergence analysis:
    \begin{align}
        \label{eq:lyapunov}
        Q_t &:= f(\bx{t}{0}) - f^* + \frac{8\alpha_t L^2\norm{V}^2}{n(1-\gamma^2)} \norm{\e_t^0}^2.
    \end{align}
    
    The roadmap of the preliminary analysis goes as follows: 
    \begin{enumerate}
        \item We present Lemmas \ref{lem:rr} and \ref{lem:bounded_var} that help bound the corresponding terms in the remaining analysis. Specifically, Lemma \ref{lem:rr} (\cite[Lemma 1]{mishchenko2020random}) is a characterization of the variance when sampling a number of vectors without replacement from a finite set of vectors. Meanwhile, Lemma \ref{lem:bounded_var} characterizes the averaged variance of the gradient $\nabla f_i(x)$ among all the agents. 
        \item Given that the term $\normi{\e_t^0}^2$ is critical to the analysis from  \eqref{eq:lyapunov}, 
        we consider the epoch-wise error as in Lemma \ref{lem:re_epoch}. Then, the recursion between $\normi{\e_{t+1}^0}^2$ and $\normi{\e_t^0}^2$  is obtained in Lemma \ref{lem:cons_e}, which further guides us to construct the recursion of the term $\cL_t$ defined in \eqref{eq:cLt}.
        \item The approximate descent property is derived in Lemma \ref{lem:descent}, and the recursion of the Lyapunov function $Q_t$ is then constructed in Lemma \ref{lem:final_re} by combining the recursions derived above.
    \end{enumerate}

    We first define some constants to ease the presentation in the following analysis.
    \begin{equation}
        \label{eq:Cs}
        \begin{aligned}
            \C_1&:= \frac{(m+1)(1-\gamma^2)}{3m} + \frac{3\C_4}{2},\ \C_2:= \frac{\brk{1- \prt{\frac{1+\gamma^2}{2}}^m}\C_4}{1-\gamma^2},\\ 
            \C_3&:= 12\C_4 + \C_1, \ \C_4:= \norm{V^{-1}}^2 \norm{V}^2\norm{\hLambda_a}^2.
        \end{aligned}
    \end{equation}

    Lemma \ref{lem:rr} below characterizes the variance when sampling a number of vectors without replacement from a finite set of vectors as stated in \cite[Lemma 1]{mishchenko2020random}.
    
	\begin{lemma}
		\label{lem:rr}
		Let $X_1, \ldots, X_m \in \R^p$ be fixed vectors, $\bar{X} := \frac{1}{m}\sum_{j=1}^{m}X_{j}$ be the average and $\sigma^2:=\frac{1}{m}\sum_{j=1}^{m}\normi{X_j-\bar{X}}^2$ be the population variance. For $\ell \in [m-1]$, let $X_{\pi_0},\ldots,X_{\pi_{\ell-1}}$ be sampled uniformly without replacement from $\{X_1,\ldots,X_m\}$ and $\bar{X}_{\pi}$ be their average. Then, $\E\brki{\normi{\bar{X}_{\pi}-\bar{X}}^2}=\prti{m-\ell}/\prti{\ell(m-1)}\sigma^2$.
	\end{lemma}

    Lemma \ref{lem:bounded_var} is based on Assumption \ref{as:comp_fun} and used for bounding the  averaged variances of the stochastic gradients  obtained by RR updates among all the agents. In addition, it generalizes the so-called bounded gradient dissimilarity assumption in the distributed setting; see, e.g., \cite{zhang2021fedpd,huang2023distributed}.
    \begin{lemma}
        \label{lem:bounded_var}
        Let Assumptions \ref{as:comp_fun} hold.  We have for any $x\in\R^p$,
        \small
        \begin{align*}
            \frac{1}{mn}\sum_{i=1}^n\sum_{\ell=1}^m\norm{\nabla f_{i,\ell}(x) - \nabla f_i(x)}^2\leq 2L\prt{f(x) -  f^*} + 2L\sigfmn,
        \end{align*}\normalsize
        where $\sigfmn:= f^* - \fmn$ and $f^*_{i,\ell}:=\inf_{x\in\R^p} f_{i,\ell}(x)$.
    \end{lemma}
    \begin{proof}
        The proof is finished by noting $\normi{\nabla f_{i,\ell}(x)}^2\leq 2L\prti{f_{i,\ell}(x) - f_{i,\ell}^*}$ and $\frac{1}{m}\summ\normi{\nabla f_{i,\ell}(x)-\nabla f_{i}(x)}^2 
    \leq \frac{1}{m}\summ\normi{\nabla f_{i,\ell}(x)}^2$.
    \end{proof}
    
    We are now ready to derive the three critical recursions for the corresponding error terms. 
    
    In light of Lemma \ref{lem:re}, we have addressed the aforementioned technical challenge on handling consecutive shuffled gradients. However, working with the recursion between $\normi{\e_t^{\ell + 1}}^2$ and $\normi{\e_t^\ell}^2$ according to the unified form \eqref{eq:abc}
    would result in a similar convergence rate for the consensus error to that of \cite[Lemma 14]{huang2023drr}. Such undesirable outcome is caused by the bias in $\nabla \fp{i}{\ell}(\xitl)$ relative to $\nabla f_i(\xitl)$, given the past iterates $\crk{x_{i,t}^{\ell - 1}}$. Specifically, we have $\E[\nabla \fp{i}{\ell}(\xitl)| \cF_t^{\ell - 1}]\neq \nabla f_i(\xitl)$ in general, where the filtration $\cF_t^{\ell}$ ($\ell=0,1,\ldots,m-1$) is generated by $\crki{x_{i,p}^{j}| i\in[n], j=0,1,\ldots, \ell,p=0,1,\ldots,t}$. However, we observe that      
    \begin{align}
        \label{eq:epoch_unb} 		         
        \condE{\nabla \fp{i}{\ell}(x_{i,t}^0)}{\cF_t^0} = \nabla f_i(x_{i,t}^0). 
    \end{align}      
    This motivates us to consider the epoch-wise error in Lemma \ref{lem:re_epoch} below. It is worth noting that a similar observation was made in \cite{ying2018stochastic}.

    \begin{lemma}
		\label{lem:re_epoch}
		Let Assumptions \ref{as:graph}, \ref{as:comp_fun}, and \ref{as:abc} hold. 
        We have
            \small
		\begin{equation}
		    \label{eq:cons_epoch}
            \begin{aligned}
                &\e_t^{\ell + 1} =\Gamma^{\ell + 1}\e_t^0\\
                & - \alpha_t\sum_{j=0}^{\ell} \Gamma^{\ell - j}V^{-1}\begin{pmatrix}
                    \hLambda_a \hU^{\T}\prt{\nabla \Fp{j}(\x_t^{j}) - \nabla \Fp{j}(\Bxt{0})}\\
                    \0
                \end{pmatrix}\\
                &- \alpha_t\sum_{j=0}^{\ell} \Gamma^{\ell - j}V^{-1}\begin{pmatrix}
                    \hLambda_a \hU^{\T}\prt{\nabla \Fp{j}(\Bxt{0}) - \nabla F(\Bxt{0})}\\
                    \0
                \end{pmatrix}.
        	\end{aligned}
		\end{equation}\normalsize
    \end{lemma}
    \begin{proof}
        Unrolling the recursion in Lemma \ref{lem:re} yields the desired result.
    \end{proof}
    
    \begin{remark}
        \label{rem:cons_drr}
		It is possible to derive a consensus recursion for D-RR similar to \eqref{eq:cons_epoch}:
        \small
		\begin{equation}
			\label{eq:cons_epoch_drr}
			\begin{aligned}
				&\x_{t}^{\ell + 1} - \1(\bx{t}{\ell + 1})^{\T} = \prt{W - \frac{\1\1^{\T}}{n}}^{\ell + 1} \prt{\x_t^0 - \Bxt{0}}\\
                &- \alpha_t\sum_{j=0}^{\ell}\prt{W - \frac{\1\1^{\T}}{n}}^{\ell - j+1}\Delta^\ell_t,
			\end{aligned}
		\end{equation}\normalsize
        where $\Delta^\ell_t:= \nabla \Fp{\ell}(\x_t^{\ell}) - \nabla \Fp{\ell}(\Bxt{0}) + \nabla \Fp{\ell}(\Bxt{0}) - \nabla F(\Bxt{0}) + \nabla F(\Bxt{0})$.

        The key difference between \eqref{eq:cons_epoch_drr} and \eqref{eq:cons_epoch} lies in the last term of \eqref{eq:cons_epoch_drr}, which introduces additional disturbance about data heterogeneity due to its dependence on $\sumn\normi{\nabla f_i(\bx{t}{0})}^2$. 
        
	\end{remark}

    Lemma \ref{lem:cons_e} is then derived according to Lemma \ref{lem:re_epoch} for bounding $\condEi{\normi{\e_{t+1}^0}^2}{\cF_t^0}$.
	\begin{lemma}
		\label{lem:cons_e}
		Let Assumptions \ref{as:graph}, \ref{as:comp_fun}, and \ref{as:abc} hold. We have
        \small
		\begin{align*}
			& \condE{\norm{\e_{t+1}^0}^2}{\cF_t^0} \leq (\frac{1 + \gamma^2}{2})^m\norm{\e_t^0}^2 + \frac{3\alpha_t^2mn \C_4 L^2}{\normi{V}^2(1-\gamma^{2})}\condE{\cL_t}{\cF_t^0}\\
			&\quad + \frac{6\alpha_t^2m^2nL\C_4}{\normi{V}^2}\prt{f(\bx{t}{0}) - f^* + \sigfmn}
		\end{align*}\normalsize
        where $\cL_{t}$ is defined as 
            \begin{align}
                \label{eq:cLt}
                \cL_t:= \frac{1}{n}\sum_{\ell=0}^{m-1}\prt{\norm{V}^2\norm{\e_t^\ell}^2 + n\norm{\bx{t}{\ell} - \bx{t}{0}}^2}.
            \end{align}
	\end{lemma}

    \begin{proof}
            See Appendix \ref{app:cons_e}.
    \end{proof}


    We next derive the recursion of $\condEi{\cL_t}{\cF_t^0}$ to help eliminate the corresponding term in Lemma \ref{lem:cons_e}. 

	\begin{lemma}
		\label{lem:cL}
		Let Assumptions \ref{as:graph}, \ref{as:comp_fun}, and \ref{as:abc} hold. Suppose the stepsize $\alpha_t$ satisfy $\alpha_t \leq \min\crki{1/(2mL), \prti{1-\gamma^2}/\prti{4L\sqrt{\C_4}}}.$
		We have
		\begin{align*}
			&\condE{\cL_t}{\cF_t^0} \leq \frac{4\brk{1-(\frac{1+\gamma^2}{2})^m}\norm{V}^2}{1-\gamma^2}\frac{\norm{\e_t^0}^2}{n}\\
            &\quad + 4\alpha_t^2m^3\norm{\nabla f(\bx{t}{0})}^2 + \frac{8m^2\alpha_t^2\C_1 L}{1-\gamma^2}\prt{f(\bx{t}{0}) - f^* + \sigfmn}.
		\end{align*}
	\end{lemma}

	\begin{proof}
            See Appendix \ref{app:cL}.
	\end{proof}

    Regrading the recursion of the function values, Lemma \ref{lem:descent} is directly obtained from Assumption \ref{as:comp_fun} provided that $\bx{t + 1}{0} = \bx{t}{0} - \alpha_t / n \sumn \sum_{\ell=0}^{m-1} \nabla \fp{i}{\ell}(x_{i,t}^{\ell})$. 
    \begin{lemma}
        \label{lem:descent}
        Let Assumptions \ref{as:graph}, \ref{as:comp_fun}, and \ref{as:abc} hold. Set the stepsize $\alpha_t\leq 1/(mL)$. We have
            \begin{align}
                f(\bar{x}_{t+1}^{0})&\leq {f(\bar{x}_{t}^{0})}-\frac{\alpha_t{m}}{2}{\norm{{\nabla}f(\bar{x}_t^{0})}^2}+{\alpha_t}{L}^2\cL_{t}.
            \end{align}
    \end{lemma} 
    
    \begin{proof}
        According to Remark \ref{rem:cons2e} that $\normi{\x_t^{\ell} - \Bxt{\ell}}^2\leq \normi{V}^2\normi{\e_t^\ell}^2$, the result is obtained following \cite[Lemma 13]{huang2023drr}.
    \end{proof}

    Finally, Lemma \ref{lem:final_re} presents the recursion for the Lyapunov function $Q_t$ by combining Lemmas \ref{lem:cons_e}-\ref{lem:descent}. 
        \begin{lemma}
            \label{lem:final_re}
            Let Assumptions \ref{as:graph}, \ref{as:comp_fun}, and \ref{as:abc} hold. Set the stepsize $\alpha_t$ satisfy 
            \begin{align*}
                \alpha_t\leq \min\crk{\frac{1}{4\sqrt{2} mL}, \frac{\sqrt{1-\gamma^2}}{6m^{\frac{3}{4}}\C_4^{\frac{1}{4}} L}, \frac{1-\gamma^2}{2\sqrt{mL^2\C_1}}, \frac{1-\gamma^2}{2\sqrt{6m\C_4} L}}.
            \end{align*}
            We have
            \begin{align*}
                &\condE{Q_{t + 1}}{\cF_t^0}\leq \prt{1 + \frac{8\alpha_t^3m^2 L^3\C_1}{1-\gamma^2} + \frac{96\alpha_t^3m^2L^3\C_4}{1-\gamma^2}}Q_t\\
                &\quad - \frac{\alpha_t m }{4}\norm{\nabla f(\bx{t}{0})}^2+ \frac{8\alpha_t^3m^2L^3\C_3 \sigfmn}{1-\gamma^2}
            \end{align*}
            where the constants $\C_i$'s are defined in \eqref{eq:Cs}.
        \end{lemma}
        \begin{proof}
            See Appendix \ref{app:final_re}.
        \end{proof}

        \begin{remark}
            Generally speaking, in this work we are able to enhance the convergence guarantees for GT-RR and ED-RR compared to that of D-RR through the following four areas of improvements.
    
    Firstly, we consider the epoch-wise error, as described in Lemma \ref{lem:re_epoch}, which allows us to leverage the unbiasedness property \eqref{eq:epoch_unb} and leads to Lemma \ref{lem:cons_e} that bounds $\condEi{\normi{\e_{t+1}^0}^2}{\cF_t^0}$. Such treatment eliminates the coefficient $1/(1-\gamma^2)$ multiplying the term $(f^* - \fmn)$ in comparison with \cite[Lemma 14]{huang2023drr}. 
    
    Secondly, we conduct a more careful analysis as in Lemma \ref{lem:cL}, which results in the improved Lyapunov function $Q_t$ given in \eqref{eq:lyapunov}. Specifically, we consider the factor $1/(1-\gamma^2)$ multiplied with $\normi{\e_t^0}^2$ instead of $1/(1-\gamma^2)^2$ in the analysis of D-RR.
    
    Thirdly, we utilize the randomness introduced by RR updates as in Lemma \ref{lem:cL} in light of Lemma \ref{lem:rr} to improve the dependency of $\cL_t$ (defined in \eqref{eq:cLt}) on $m$. These two results together give rise to a better bound on $\condEi{Q_{t + 1}}{\cF_t^0}$.
    
    The above three areas of refinements are solely due to technical improvements. 
    Lastly, the algorithmic structure of GT-RR and ED-RR  contributes to the final improvement when compared to D-RR, as indicated in Remark \ref{rem:cons_drr}. In particular, the presence of the extra term $\sumn\normi{\nabla f_i(\bx{t}{0})}^2$ in D-RR degenerates its convergence rate. 

        \end{remark}

    \section{Main Results}
    \label{sec:main}

    This section presents the main theorems on the convergence rates of GT-RR and ED-RR under smooth nonconvex objective functions in Subsection \ref{subsec:nonconvex}, and the convergence rates under smooth objective functions satisfying the PL condition in Subsection \ref{subsec:PL}. Both results are established directly based on Lemma \ref{lem:final_re}.
    Additionally, we introduce an enhanced convergence rate for D-RR in Theorem \ref{thm:drr} for smooth nonconvex objective functions, which implies that our analysis techniques extend beyond those in \cite{huang2023drr}.

    \subsection{General Nonconvex Case}
    \label{subsec:nonconvex}

        In this part, we first introduce the convergence result under the unified framework given in Theorem \ref{thm:nonconvex_final_results}, which involves some unspecified network-related parameters. Then, the convergence results for GT-RR and ED-RR are presented in Corollary \ref{cor:nonconvex_DGT_RR} and Corollary \ref{cor:nonconvex_ED_RR}, respectively. Moreover, we establish an improved convergence rate for D-RR in Theorem \ref{thm:drr}.

        \begin{theorem}
        \label{thm:nonconvex_final_results}
        Let Assumptions \ref{as:graph}, \ref{as:comp_fun}, and \ref{as:abc} hold. Let the stepsize $\alpha_t = \alpha$ satisfy
        \small
         \begin{equation}
            \label{eq:alpha_ncvx}
            \begin{aligned}
               \alpha \leq \min&\crk{\prt{\frac{1-\gamma^2}{192 m^2 L^3\C_1 T}}^{\frac{1}{3}},\frac{1}{4\sqrt{2} mL}, \frac{\sqrt{1-\gamma^2}}{6m\C_4^{\frac{1}{4}} L},\right.\\
               &\left. \frac{1-\gamma^2}{2\sqrt{mL^2\C_1}}, \frac{1-\gamma^2}{2\sqrt{6m\C_4} L}}.
           \end{aligned}
         \end{equation}\normalsize
        Then, for the unified update \eqref{eq:abc_epoch} we have
        \small
        \begin{equation}
            \label{eq:nonconvex}
            \begin{aligned}
                &\min_{t=0,\ldots,T-1}\E\brk{\norm{{\nabla}f(\bar{x}_t^{0})}^2}
                \leq \frac{12(f(\bar{x}_0^0)-f^*)}{\alpha mT} + \frac{32\alpha^2m L^3\C_3}{1-\gamma^2}\sigfmn\\
                &+ \frac{96 L^2\C_4\C_b}{\norm{\hLambda_a}^2 mn(1-\gamma^2)T} + \frac{384 \alpha^2 L^3\C_4\norm{\hLambda_b^{-1}}^2[f(\bar{x}_0^0) - \frac{1}{n}\sumn f_i^*]}{m(1-\gamma^2)T},
            \end{aligned}
        \end{equation}\normalsize
        where $\C_1$, $\C_3$, $\C_4$ are given in \eqref{eq:Cs}, $\C_b:= \normi{\x_0^0 - \1(\bar{x}_0^0)^{\T}}^2 + 2\normi{\hLambda_b^{-1}}^2\normi{(I-W)\x_0^0}^2$, and $f_i^*:= \inf_{x\in\R^p} f_i(x)$.

        In addition, if we set $\beta \geq 2\sqrt{2}(1-\gamma^2) + 3\prt{\frac{(1-\gamma^2)^2\C_1}{m}}^{\frac{1}{4}} + \sqrt{\frac{\C_1}{m}} + \sqrt{\frac{6\C_4}{m}}$ and $\alpha = \frac{1}{\frac{2mL\beta}{1-\gamma^2} + \prt{\frac{192 m^2 L^3 \C_1 T}{1-\gamma^2}}^{\frac{1}{3}}},$
        then,
        \small
        \begin{equation}
            \label{eq:abc_rr}
            \begin{aligned}
                &\min\limits_{t=0,\ldots,T-1}\E\brk{\norm{{\nabla}f(\bar{x}_t^{0})}^2}
                \leq \frac{24L\beta(f(\bar{x}_0^0)-f^*)}{(1-\gamma^2)T}\\
                &+ \frac{72L\C_1^{\frac{1}{3}}(f(\bar{x}_0^0)-f^*)}{(1-\gamma^2)^{\frac{1}{3}}m^{\frac{1}{3}} T^{\frac{2}{3}}} + \frac{L\C_3\sigfmn}{\C_1^{\frac{2}{3}}(1-\gamma^2)^{\frac{1}{3}}m^{\frac{1}{3}} T^{\frac{2}{3}}}\\
                &+ \frac{12 L\C_4\norm{\hLambda_b^{-1}}^2[(f(\bar{x}_0^0)- \frac{1}{n}\sumn f_i^*)]}{\C_1^{\frac{2}{3}}m^{\frac{7}{3}}(1-\gamma^2)^{\frac{1}{3}}T^{\frac{5}{3}}}+ \frac{96 L^2\C_4\C_b}{\normi{\hLambda_a}^2 mn(1-\gamma^2)T}.
            \end{aligned}
        \end{equation}\normalsize
        
        \end{theorem}
                \begin{proof}
            See Appendix \ref{app:nonconvex_final_results}.
        \end{proof}
        

        \begin{remark}
        \label{rem:ini}
        Based on Lemma \ref{lem:abc}, for both GT-RR and ED-RR we have $\s_0^0 = (W-I)\x_0^0 + \alpha_0 W\nabla F(\1(\bar{x}_0^0)^{\T})$. Hence the term $\norm{\e_0^0}^2$ is the same for both algorithms and can be bounded as follows:
        \small
        \begin{equation}
            \label{eq:e0}
            \begin{aligned}
                &\norm{\e_0^{0}}^2 
                \leq \norm{V^{-1}}^2\left[\C_b+ 4\alpha_0^2 nL\norm{\hLambda_b^{-1}}^2\prt{f(\bar{x}_0^0) - \frac{1}{n}\sumn f_i^*}\right].
            \end{aligned}
        \end{equation}\normalsize
        \end{remark}

        We now apply Theorem \ref{thm:nonconvex_final_results} to GT-RR and ED-RR and present their convergence results in Corollary \ref{cor:nonconvex_DGT_RR} and Corollary \ref{cor:nonconvex_ED_RR}, respectively.

        \begin{corollary}[GT-RR]
            \label{cor:nonconvex_DGT_RR}
            Suppose the conditions in Theorem \ref{thm:nonconvex_final_results} hold. Let $A=W$, $B=(I-W)$, $C=W$, and set $\beta = \beta_1 = 2\sqrt{2}(1-\lambda^2) + 3\prt{\frac{42(1-\lambda^2)^2}{m}}^{\frac{1}{4}} + \sqrt{\frac{42}{m}}+ \sqrt{\frac{162}{m}}.$

            Then, the iterates generated by GT-RR satisfy
            \small
            \begin{align*}
                &\min\limits_{t=0,\ldots,T-1}\E\brk{\norm{{\nabla}f(\bar{x}_t^{0})}^2}
                \leq \frac{24L\beta_1(f(\bar{x}_0^0)-f^*)}{(1-\lambda^2)T}\\
                &\quad + \frac{674 L[f(\bar{x}_0^0) - \frac{1}{n}\sumn f_i^*]}{m^{\frac{7}{3}}(1-\lambda^2)^{\frac{7}{3}}T^{\frac{5}{3}}}+ \frac{252L(f(\bar{x}_0^0)-f^*)}{(1-\lambda^2)^{\frac{1}{3}}m^{\frac{1}{3}} T^{\frac{2}{3}}}\\
                & + \frac{762L\sigfmn}{(1-\lambda^2)^{\frac{1}{3}}m^{\frac{1}{3}} T^{\frac{2}{3}}} + \frac{2592L^2\normi{\x_0^0 - \1(\bar{x}_0^0)^{\T}}^2}{mn(1-\lambda^2)T}\\
                & + \frac{5184 L^2\normi{(I-W)\x_0^0}^2}{mn(1-\lambda^2)^3T}.
            \end{align*}\normalsize
            
        \end{corollary}

        \begin{proof}
            See Appendix \ref{app:nonconvex_DGT_RR}.
        \end{proof}

        \begin{corollary}[ED-RR]
            \label{cor:nonconvex_ED_RR}
            Suppose the conditions in Theorem \ref{thm:nonconvex_final_results} hold and $W$ is positive definite. Let $A=W,B=(I-W)^{1/2},C=I$, and set $\beta = \beta_2 =  2\sqrt{2}(1-\lambda) + 3\prti{\frac{38(1-\lambda)^2}{3\underline{\lambda}m}}^{1/4} + \sqrt{\frac{38}{3\underline{\lambda}m}} + \sqrt{\frac{48}{\underline{\lambda}m}}.$

            Then, then the iterates generated by ED-RR satisfy
            \begin{align*}
                &\min_{t= 0,1,\ldots, T-1} \E\brk{\norm{\nabla f(\bx{t}{0})}^2}\leq \frac{24L\beta_2(f(\bar{x}_0^0)-f^*)}{(1-\lambda^2)T}\\
                &\quad + \frac{168L(f(\bx{0}{0}) - f^*)}{\underline{\lambda}^{\frac{1}{3}}(1-\lambda)^{\frac{1}{3}}m^{\frac{1}{3}} T^{\frac{2}{3}}} +  \frac{768L^2\norm{\x_0^0 - \1(\bar{x}_0^0)^{\T}}^2}{mn\underline{\lambda}(1-\lambda)T}\\
                &\quad + \frac{1536L^2\brk{\norm{(I-W)\x_0^0}^2}}{mn\underline{\lambda}(1-\lambda)^2T} + \frac{227L\sigfmn}{\underline{\lambda}^{\frac{1}{3}}(1-\lambda)^{\frac{1}{3}}m^{\frac{1}{3}} T^{\frac{2}{3}}}\\
                &\quad + \frac{200 L[f(\bar{x}_0^0) - \frac{1}{n}\sumn f_i^*]}{\underline{\lambda}^{\frac{1}{3}}m^{\frac{7}{3}}(1-\lambda)^{\frac{4}{3}}T^{\frac{5}{3}}}
            \end{align*}
            where $\underline{\lambda}$ is the smallest eigenvalue of $W$.
            
        \end{corollary}

        \begin{proof}
            See Appendix \ref{app:nonconvex_ED_RR}.            
        \end{proof}

        \begin{remark}
            If we initialize all the agents' iterates at the same point, i.e., $\x_0^0=\1(x_0)^{\T}$ for some $x_0 \in \R^p$, and set $W$ to be $(1-\tau)W + \tau I$ (with a slight abuse of notation) for some $\tau\in (0,1)$, then the iterates generated by GT-RR satisfy the simplified inequality below:
            \small
            \begin{align*}
                &\min_{t= 0,1,\ldots, T-1} \E\brk{\norm{\nabla f(\bx{t}{0})}^2}
                \leq \order{\frac{f(x_{0}) - f^*}{(1-\lambda)^{\frac{1}{3}}m^{\frac{1}{3}} T^{\frac{2}{3}}}\right.\\
                &\left. +\frac{f^* - \frac{1}{n}\sumn f_i^*}{m^{\frac{7}{3}}(1-\lambda)^{\frac{7}{3}}T^{\frac{5}{3}}} + \frac{f(x_0) - f^*}{(1-\lambda)T} + \frac{\sigfmn}{(1-\lambda)^{\frac{1}{3}}m^{\frac{1}{3}} T^{\frac{2}{3}}}}.
            \end{align*}\normalsize
            Similarly, the iterates generated by ED-RR satisfy
            \small
            \begin{align*}
                &\min_{t= 0,1,\ldots, T-1} \E\brk{\norm{\nabla f(\bx{t}{0})}^2}
                \leq \order{\frac{f(x_{0}) - f^*}{(1-\lambda)^{\frac{1}{3}}m^{\frac{1}{3}} T^{\frac{2}{3}}}\right.\\
                &\left. +\frac{f^* - \frac{1}{n}\sumn f_i^*}{m^{\frac{7}{3}}(1-\lambda)^{\frac{4}{3}}T^{\frac{5}{3}}} + \frac{f(x_0) - f^*}{(1-\lambda)T} + \frac{\sigfmn}{(1-\lambda)^{\frac{1}{3}}m^{\frac{1}{3}} T^{\frac{2}{3}}}}.
            \end{align*}\normalsize
            These results are corresponding to those listed in Table \ref{tab:summary_ncvx}.

            In addition, for both GT-RR and ED-RR, the sample complexity $mT$ required by each agent to ensure $\min_{t= 0,1,\ldots, T-1} \E\brki{\normi{\nabla f(\bx{t}{0})}^2}\leq \varepsilon^2$ is $\orderi{\frac{\sqrt{m }}{\varepsilon^3\sqrt{1-\lambda}}+ \frac{m }{\varepsilon^2 (1-\lambda)}}$ for small enough $\varepsilon$, that is, $\varepsilon\leq \orderi{m^{7/5}\sqrt{1-\lambda}}$. Such a sample complexity is comparable to that of the centralized RR algorithm in \cite{mishchenko2020random}, up to the network topology coefficient $1/(1-\lambda)$.
        \end{remark}


        In Theorem \ref{thm:drr} below, we present an improved convergence result for D-RR, which serves to corroborate the superiority of the analysis in this work compared to those in \cite{huang2023drr}.

        \begin{theorem}
            \label{thm:drr}
            Let Assumptions \ref{as:graph} and \ref{as:comp_fun} hold, and let the stepsize $\alpha_t = \alpha$ satisfy 
            $\alpha\leq \min\crki{\prti{1-\lambda^2}/\prti{4mL\sqrt{14}}, \brki{\prti{(1-\lambda^2)^2}/\prti{384m^2L^3 T}}^{\frac{1}{3}}}.$

            We have for D-RR in \cite{huang2023drr} that
            \begin{align*}
                &\min_{t = 0,1,\ldots, T-1}\E\brk{\norm{\nabla f(\bx{0}{0})}^2} \leq \frac{12\prt{f(\bx{0}{0}) - f^*}}{\alpha m T}\\
                &\quad + \frac{96L^2\norm{\x_0^0 - \1(\bar{x}_0^0)^{\T}}^2}{mn(1-\lambda^2)T}+ \frac{984\alpha^2 m L^3 \sigfmn}{(1-\lambda^2)^2}.
            \end{align*}

            In addition, if the stepsize is set as $\alpha = 1/\crki{\prti{4\sqrt{14}mL}/\prti{1-\lambda^2} + \brki{\prti{384m^2 L^3 T}/\prti{(1-\lambda^2)^2}}^{1/3}},$
            then the iterates generated by D-RR satisfy
            \begin{align*}
                &\min_{t = 0,1,\ldots, T-1}\E\brk{\norm{\nabla f(\bx{t}{0})}^2} \leq \frac{88L(f(\bx{0}{0}) - f^*)+20L\sigfmn}{(1-\lambda^2)^{\frac{2}{3}} m^{\frac{1}{3}}T^{\frac{2}{3}}}\\
                &+ \frac{96L^2\norm{\x_0^0 - \1(\bar{x}_0^0)^{\T}}^2}{mn(1-\lambda^2)T}+ \frac{48\sqrt{14} L(f(\bar{x}_0^0) - f^*)}{(1-\lambda^2)T}.
            \end{align*}
        \end{theorem}
        \begin{proof}
            See Appendix \ref{app:thm_drr}.
        \end{proof}
        \begin{remark}
            \label{rem:drr}
            It is noteworthy that D-RR can attain a convergence rate of $\orderi{1/[m^{1/3}(1-\lambda)^{2/3}T^{2/3}]}$ for large $T$ according to Theorem \ref{thm:drr}, which improves upon the previous convergence rate of $\orderi{1/[(1-\lambda) T^{2/3}]}$ given in \cite[Theorem 4]{huang2023drr}.
            The improved convergence rate is evident in terms of the sample size $m$ and the coefficient $(1-\lambda)$ related to the network topology.
        \end{remark}

        We summarize the implications and significance of the findings stated in this section. Theorems \ref{thm:nonconvex_final_results} and \ref{thm:drr} present the first results indicating that distributed random reshuffling methods over networks can achieve convergence rates comparable to that of the centralized RR method, i.e., $\orderi{1/(m^{1/3}T^{2/3})}$, concerning the sample size $m$ and the epoch number $T$. Notably, such a result outperforms that of SGD when high accuracy is demanded. 
        In addition, Corollaries \ref{cor:nonconvex_DGT_RR} and \ref{cor:nonconvex_ED_RR} highlight the enhanced convergence rate of $\orderi{1/[(1-\lambda)^{1/3}m^{1/3}T^{2/3}]}$ for GT-RR and ED-RR, which relieves the impact of the network topology compared to D-RR.

\subsection{PL Condition Case}
\label{subsec:PL}

    In this part, we introduce the convergence results of GT-RR and ED-RR under the PL condition. Similar to the general nonconvex case, we first present the convergence result under the unified framework in Theorem \ref{thm:linear}. Then, the specific convergence results for GT-RR and ED-RR are outlined in Corollary \ref{cor:PL_DGT_RR} and Corollary \ref{cor:PL_ED_RR}, respectively. In contrast to the recursion given in Lemma \ref{lem:cons_e}, we introduce a new recursion for the term $\normi{\e_t^0}^2$ in Lemma \ref{lem:cons_e_m} that retains a coefficient $\C(\gamma, m):=\sum_{j=0}^{m-1}(\gamma^2)^{m-1-j}$. This modification allows us to avoid an additional $1/(1-\gamma^2)$ term after the unrolling process, resulting in an improved order of $m$.

    \begin{lemma}
		\label{lem:cons_e_m}
		Let Assumptions \ref{as:graph}, \ref{as:comp_fun}, and \ref{as:abc} hold. Let the stepsize $\alpha_t$ satisfy $\alpha_t \leq \min\crki{\prti{1-\gamma^2}/\prti{\sqrt{24mL^2 \C_1}}, \prti{1-\gamma^2}/\prti{\sqrt{4m^2L^2\C_1}}}.$
        Then, 
        \small
		\begin{align*}
            &\condE{\norm{\e_{t+1}^{0}}^2}{\cF_t^0} \leq q_0\norm{\e_t^0}^2+  \frac{24\alpha_t^4m^4n L^3\C_4}{\normi{V}^2(1-\gamma^2)}\brk{f(\bx{t}{0}) - f^*}\\
            &+ \frac{6\alpha_t^2mnL\C_4\prt{\C(\gamma, m) + 1}}{\normi{V}^2}\brk{f(\bx{t}{0}) - f^* + \sigfmn},
        \end{align*}\normalsize
        where $\C(\gamma, m):=\sum_{j=0}^{m-1}(\gamma^2)^{m-1-j}$ and $q_0:= \crki{1 + [(1+\gamma^2)/2]^m}/2$.
	\end{lemma}
    
    \begin{proof}
        See Appendix \ref{app:cons_e_m}.
    \end{proof}
    
    The convergence results under the unified framework, incorporating both a decreasing stepsize $\alpha_t = \theta/[\mu m(t+K)]$, as well as a constant stepsize $\alpha_t = \alpha$, are presented in Theorem \ref{thm:linear}. These results are established when the objective functions satisfy the PL condition.

    \begin{theorem}
        \label{thm:linear}
        Suppose Assumptions \ref{as:graph}-\ref{as:abc} are valid. 
        If we consider a constant stepsize $\alpha_t = \alpha$ satisfying 
        \small
        \begin{align*}
            \alpha\leq \min&\crk{\sqrt{\frac{\mu(1-\gamma^2)}{768 mL^3 \C_1}}, \frac{1-\gamma^2}{\sqrt{24m^2L^2\C_1}},{\frac{1-\gamma^2}{m\mu}},\frac{\sqrt{1-\gamma^2}}{6m^{\frac{3}{4}}\C_1^{\frac{1}{4}} L}},
        \end{align*}\normalsize
        then for the unified update \eqref{eq:abc_epoch}, we have
        \begin{align*}
            &\frac{1}{n}\sumn \E\brk{f(x_{i,t}^0) - f^*}
            \leq \frac{L\norm{V}^2}{n}q_0^{t} \norm{\e_0^0}^2\\
            &\quad + \brk{2\prt{1 - \frac{\alpha\mu m}{4}}^{t} + \frac{48\alpha^2 mL^2\C_4}{1-\gamma^2} + \frac{96\alpha^4m^4 L^4 \C_4}{(1-\gamma)^2} }Q_0\\
            &\quad + \frac{16\alpha^2 m L^2\kappa}{1-\gamma^2}\brk{(4\C_1 + 69\C_4) + \frac{42\alpha^2 m^3 L\mu \C_4}{1-\gamma^2}}\sigfmn.
        \end{align*}

        If decreasing stepsize $\alpha_t = \theta/(m\mu (t + K))$ is considered, where $\theta > 16$, $\kappa:= L/\mu$, and
        \small
        \begin{align*}
            K\geq \max&\crk{\sqrt{\frac{768\theta^2 \C_1\kappa^3}{m(1-\gamma^2)}}, \sqrt{\frac{12\theta^2\kappa^2\C_1}{(1-\gamma^2)^2}}, {\frac{2\theta}{1-\gamma^2}}, \frac{6\theta \C_1^{\frac{1}{4}}\kappa}{m^\frac{1}{4}\sqrt{1-\gamma^2}}},
        \end{align*}\normalsize
        then
        \small
        \begin{align*}
            &\frac{1}{n}\sumn \E\brk{f(x_{i,t}^0) - f^*} 
            \leq \frac{\norm{V}^2 L }{n}q_0^{t}\norm{\e_{0}^0}^2 + \brk{2\prt{\frac{K}{t + K}}^{\frac{\theta}{4}}\right.\\
            &\left.+ \frac{144\theta^2\kappa^2 \C_4}{m(1-\gamma^2)(t + K)^2} + \frac{96\theta^4\kappa^4 \C_4}{(1-\gamma^2)^2 (t + K)^4}}Q_0\\
            &+ \brk{8\kappa\C_3 + 21\C_4 + \frac{21\kappa^2 \theta^2\C_4}{(1-\gamma^2)(t + K)^2}}\frac{32\theta^2 \kappa^2\sigfmn}{m(1-\gamma^2)(t + K)^2},
        \end{align*}\normalsize
        where $\normi{\e_{0}^0}^2$ is upper bounded in \eqref{eq:e0}.
    \end{theorem}
        \begin{proof}
        See Appendix \ref{app:linear}.
    \end{proof}

    \begin{remark}
        According to Theorem \ref{thm:linear}, the expected error term, i.e., $1/n\sumn \E\brki{f(x_{i,t}^0) - f^*}$, decreases exponentially fast to a neighborhood of zero with size $\orderi{m\alpha^2/(1-\gamma^2)}$ when a constant stepsize $\alpha$ is used. Moreover, the error term converges to zero at a rate of $\orderi{1/[m(1-\gamma^2)T^2]}$ under the decreasing stepsize policy.
    \end{remark}


    The specialized convergence results for GT-RR and ED-RR are presented below. The proofs are omitted as they are straightforward when the constants related to the choice of $A, B,$ and $C$ are specified for GT-RR in Corollary \ref{cor:nonconvex_DGT_RR} and for ED-RR in Corollary \ref{cor:nonconvex_ED_RR}.
    
    \begin{corollary}[GT-RR]
        \label{cor:PL_DGT_RR}
        Let the conditions in Theorem \ref{thm:linear} hold. Let $A=W$, $B=(I-W)$, $C=W$. 
        Under a constant stepsize $\alpha_t = \alpha$,
        for GT-RR we have
        \small
        \begin{align*}
            &\frac{1}{n}\sumn \E\brk{f(x_{i,t}^0) - f^*}\\
            &\leq \brk{2\prt{1 - \frac{\alpha\mu m}{4}}^{t} + \frac{2016\alpha^2 mL^2}{1-\lambda^2} + \frac{4032\alpha^4m^4 L^4 \C_4}{(1-\lambda^2)^2} }Q_0\\
            &\quad + \frac{3L q_0^{t} \norm{\e_0^0}^2}{n}+ \frac{16\alpha^2 m L^2\kappa}{1-\lambda^2}\brk{2031 + \frac{1134\alpha^2 m^3 L\mu }{1-\lambda^2}}\sigfmn.
        \end{align*}\normalsize

        Under decreasing stepsize $\alpha_t = \theta/(m\mu (t + K))$, 
        for GT-RR we have
        \small
        \begin{align*}
            &\frac{1}{n}\sumn \E\brk{f(x_{i,t}^0) - f^*} 
            \leq \frac{\norm{V}^2 L }{n}q^{t}\norm{\e_{0}^0}^2 + \brk{2\prt{\frac{K}{t + K}}^{\frac{\theta}{4}}\right.\\
            &\left. + \frac{3888\theta^2\kappa^2 }{m(1-\lambda^2)(t + K)^2} + \frac{2592\theta^4\kappa^4 }{(1-\lambda^2)^2 (t + K)^4}}Q_0\\
            &+ \brk{2928\kappa + 567 + \frac{567 m \kappa^2 \theta^2}{(1-\lambda^2)(t + K)^2}}\frac{32\theta^2 \kappa^2\sigfmn}{m(1-\lambda^2)(t + K)^2}.
        \end{align*}\normalsize
        
    \end{corollary}

        

    \begin{corollary}[ED-RR]
    \label{cor:PL_ED_RR}
        Let the conditions in Theorem \ref{thm:linear} hold. Let $A=W ,B=(I-W)^{1/2},C=I$. 
        If we use a constant stepsize $\alpha_t = \alpha$,
        then for ED-RR we have 
        \begin{align*}
            &\frac{1}{n}\sumn \E\brk{f(x_{i,t}^0) - f^*}
            \leq  \frac{L\norm{V}^2}{n}q_0^{t} \norm{\e_0^0}^2 \\
            &\quad + \brk{2\prt{1 - \frac{\alpha\mu m}{4}}^{t} + \frac{384\alpha^2 mL^2}{\underline{\lambda}(1-\lambda)} + \frac{768\alpha^4m^4 L^4}{\underline{\lambda}(1-\lambda)^2} }Q_0\\
            &\quad + \frac{16\alpha^2 m L^2\kappa}{1-\lambda}\brk{\frac{604}{\underline{\lambda}} + \frac{336\alpha^2 m^3 L\mu }{\underline{\lambda}(1-\lambda)}}\sigfmn
        \end{align*}

        If we use decreasing stepsize $\alpha_t = \theta/(m\mu (t + K))$, 
        then for ED-RR we have 
        \small
        \begin{align*}
            &\frac{1}{n}\sumn \E\brk{f(x_{i,t}^0) - f^*} 
            \leq \frac{q_0^{t} \norm{V}^2 L \normi{\e_{0}^0}^2}{n} + \brk{2\prt{\frac{K}{t + K}}^{\frac{\theta}{4}}\right.\\
            &\left.+ \frac{1152\theta^2\kappa^2 }{m\underline{\lambda}(1-\lambda)(t + K)^2} + \frac{768\theta^4\kappa^4}{\underline{\lambda}(1-\lambda)^2 (t + K)^4}}Q_0\\
            &+ \brk{\frac{870\kappa}{\underline{\lambda}} + \frac{168}{\underline{\lambda}} + \frac{168m \kappa^2 \theta^2}{\underline{\lambda}(1-\lambda)(t + K)^2}}\frac{32\theta^2 \kappa^2\sigfmn}{m(1-\lambda)(t + K)^2}.
        \end{align*}\normalsize
    \end{corollary}

    \begin{remark}
        If all the agents' iterates are initialized at the same point, i.e., $\x_0^0=\1x_0^{\T}$ for some $x_0 \in \R^p$, then we have from \eqref{eq:e0} that $Q_0 = f(x^0) - f^* + \frac{32\alpha_0^3L^3\normi{V}^2\normi{V^{-1}}^2 \normi{\hLambda_b^{-1}}^2}{1-\gamma^2}\brki{f(x^0) - \frac{1}{n}\sumn f_i^*}$
        which implies that $Q_0 \sim \orderi{f(x^0) - 1/n\sumn f_i^*}$ for GT-RR and ED-RR. If we further set $W$ to be $(1-\tau)W+\tau I$ for some $\tau\in(0,1)$, then the iterates generated by GT-RR or ED-RR with a constant stepsize $\alpha$ satisfy
       \small
        \begin{align*}
            &\frac{1}{n}\sumn \E\brk{f(x_{i,T}^0) - f^*}\leq \order{\frac{\alpha^2 m \prti{f(x^0) - f^*}}{1-\lambda} + \frac{\alpha^4 m^4\sigfmn}{(1-\lambda)^2}}\\
            &\quad + \order{\frac{\alpha^2 m \sigfmn}{1-\lambda} + \frac{\alpha^4 m^4\prti{f(x^0) - f^*}}{(1-\lambda)^2}},
        \end{align*}\normalsize
       and the iterates generated by GT-RR or ED-RR using decreasing stepsize $\alpha_t = \theta/(\mu m(t + K))$ satisfy
       \small
        \begin{align*}
            &\frac{1}{n}\sumn \E\brk{f(x_{i,T}^0) - f^*}\leq \order{\frac{f(x^0) - f^*}{m(1-\lambda) T^2} + \frac{\sigfmn}{(1-\lambda)^2 T^4}}\\
            &\quad + \order{\frac{\sigfmn}{(1-\lambda)mT^2} + \frac{f(x^0) - f^*}{(1-\lambda)^2T^4}}.
        \end{align*}\normalsize
        The above results are corresponding to those listed in Table \ref{tab:summary_PL}. 
        Moreover, for both choices of the stepsizes, the sample complexity required by each agent to ensure $1/n\sumn \E\brki{f(x_{i,t}^0) - f^*}\leq \varepsilon^2$ is in the order of $\orderi{ \frac{\sqrt{m}}{\varepsilon \sqrt{1-\lambda}}+ \frac{m}{\sqrt{\varepsilon(1-\lambda)}}}$.
    \end{remark}

    We summarize the implications and significance of the findings stated in this section. Corollaries \ref{cor:PL_DGT_RR} and \ref{cor:PL_ED_RR} demonstrate that both GT-RR and ED-RR can achieve the convergence rate of $\orderi{1/[(1-\lambda)mT^2]}$, which is comparable to that of centralized RR up to the graph-related constant $(1-\lambda)$. 
    Compared to the convergence rate $\orderi{1/[(1-\lambda)^3 mT^2]}$ of D-RR, such a result improves the dependence on $(1-\lambda)$ significantly. Notably, the term $(1-\lambda)$ can behave as $\orderi{1/n^2}$ for sparse graphs like rings. Moreover, such an enhancement is derived under the PL condition, which covers a wider range of application compared to the strongly convex case considered in D-RR \cite{huang2023drr}.

    \section{Numerical Results}
    \label{sec:sims}

        This section presents two numerical examples that illustrate the performance of the proposed GT-RR and ED-RR algorithms compared with existing methods. We evaluate these algorithms on a strongly convex problem \eqref{eq:ls} (satisfying the PL condition) and for training a neural network to verify the theoretical findings. In general, the results demonstrate the superiority of GT-RR and ED-RR over their counterparts using unshuffled stochastic gradients (DSGT and ED) for sufficiently large $T$. Unlike D-RR which relies on small stepsizes to outperform unshuffled methods, GT-RR and ED-RR work under relatively large stepsizes. We acknowledge that some codes are from \cite{qureshi2020s}. All the results are averaged over $10$ repeated runs. We denote $x_{i,t}:=x_{i,t}^0$ and $\bar{x}_t:=\bar{x}_t^0$ for simplicity.

        The considered graph topologies are shown in Fig. \ref{fig:graph}.

		\begin{figure}[htbp]
			\centering
			\subfloat[Ring graph, $n = 16$.]{\includegraphics[width=0.245\textwidth]{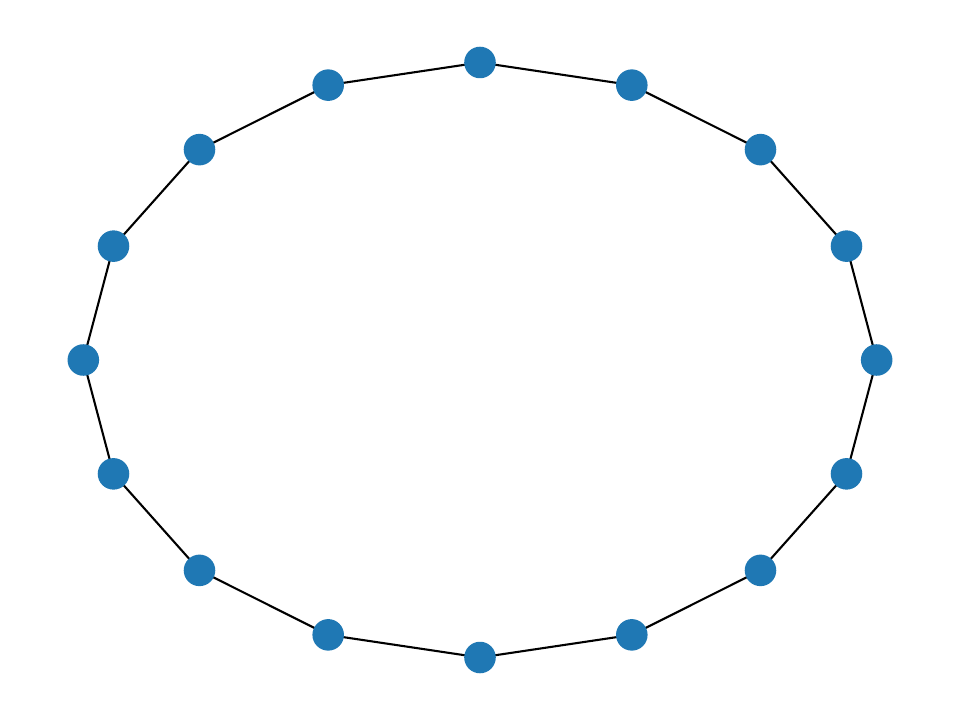}}
            \subfloat[Grid graph, $n = 16$.]{\includegraphics[width=0.245\textwidth]{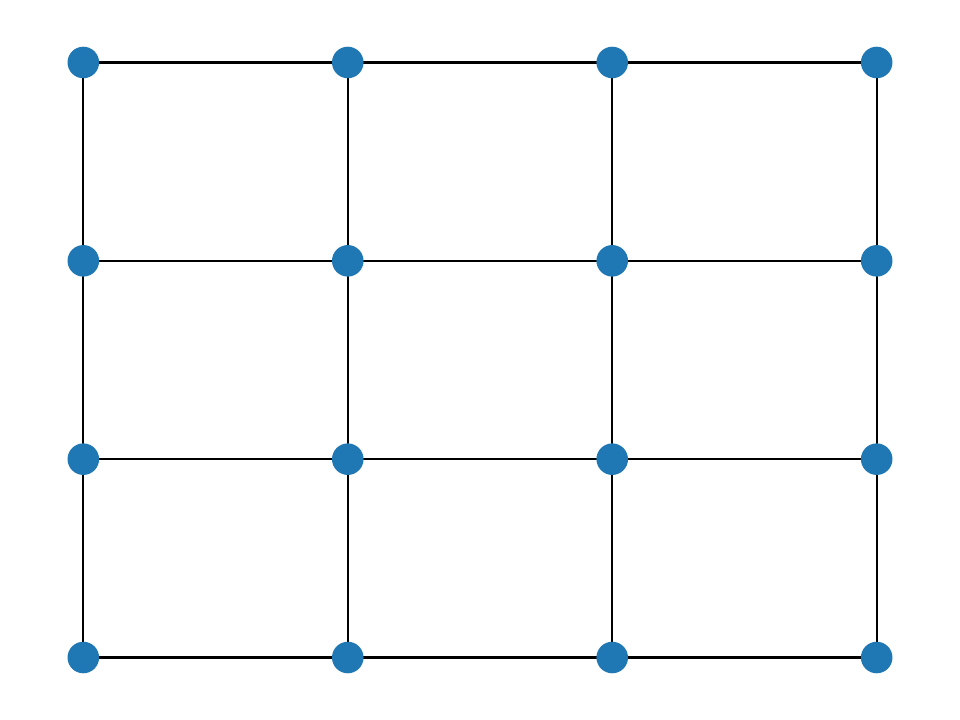}}
			\caption{Illustration of two graph topologies. The spectral gaps $(1-\lambda)$ increases from ring graph to grid graph.}
			\label{fig:graph}
		\end{figure}

        \subsection{Distributed estimation}
        \label{sec:ls}
        
        In this part, we consider a canonical distributed estimation problem with $n$ sensors connected over a network \cite{xu2017convergence}. Each sensor is designed to measure a parameter $x\in\R^p$ with some scaled Gaussian noise $i^2 s_i\in\R^p$, i.e., $b_i = A_i x + i^2 s_i$, due to the heterogeneity of sensors. Here, $A_i\in\R^{m\times p}$ is the measurement matrix and $b_i\in\R^p$ is the collected data of sensor $i$. To obtain accurate estimation of the parameter $x$, the sensors collaborate to solve the maximum likelihood estimation problem, which is equivalent to Problem \eqref{eq:ls}.
        
    \begin{equation}
        \label{eq:ls}
        \begin{aligned}
            \min_{x\in\R^p} f(x) &= \frac{1}{n}\sumn f_i(x),
            f_i(x):= \frac{1}{2m}\normi{A_i x - b_i}^2 + \frac{\rho}{2}\normi{x}^2.
        \end{aligned}
    \end{equation}
    We set $p=10$ and $m=5000$ to generate the rows $\crki{a_{ij}}_{j=1}^m$ of $A_i\in\R^{m\times p}$ from a standard multivariate normal distribution. For a given $\hat{x}\in\R^p$ (generated from a multivariate normal distribution with a fixed random seed), each $b_{i}\in\R^p$ is generated by $b_{i} = A_i\hat{x} + i^2 s_{i}$, where $s_{i}\in\R^p$ is from a standard multivariate normal distribution.

    We compare the performance of GT-RR (Algorithm \ref{alg:GT-RR}), ED-RR (Algorithm \ref{alg:ED-RR}), D-RR \cite{huang2023drr}, SGD, Exact Diffusion (ED) \cite{yuan2018exact,huang2021improving}, DSGT \cite{pu2021distributed}, and centralized RR (Algorithm \ref{alg:CRR}).
    We use both a constant stepsize (Fig. \ref{fig:ls_c}) and decreasing stepsizes (Fig. \ref{fig:ls_d}). It can be seen that all the decentralized RR methods achieve better accuracy than their unshuffled counterparts after the starting epochs. Due to the effect of graph topology, their performance is worse than centralized RR as expected.
        
        Comparing the three decentralized RR methods, we find that 
        as the network topology becomes better-connected (i.e., from a ring graph to a grid graph), their performance tends to be more comparable to that of centralized RR.
        Regarding the differences between GT-RR and ED-RR, we observe that GT-RR performs slightly better than ED-RR. For the stepsize policy, decreasing stepsizes are found to be more favorable, allowing for larger stepsizes in the initial epochs and leading to similar accuracy while requiring fewer epochs.

    \begin{figure}[htbp]
        \centering
        \subfloat[Ring graph, $n = 16$.]{\includegraphics[width = 0.245\textwidth]{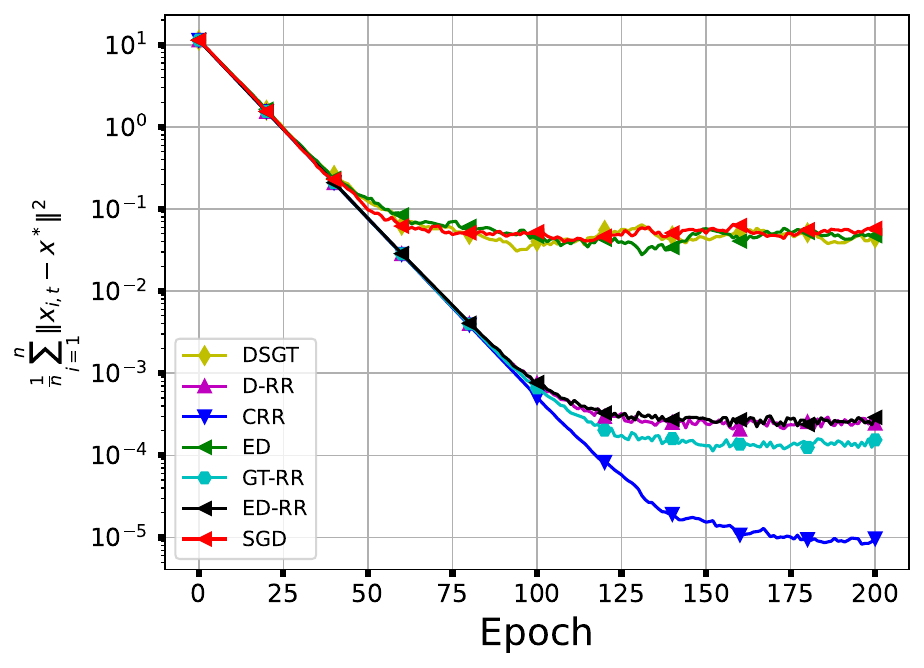}}
        \subfloat[Grid graph, $n = 16$.]{\includegraphics[width = 0.245\textwidth]{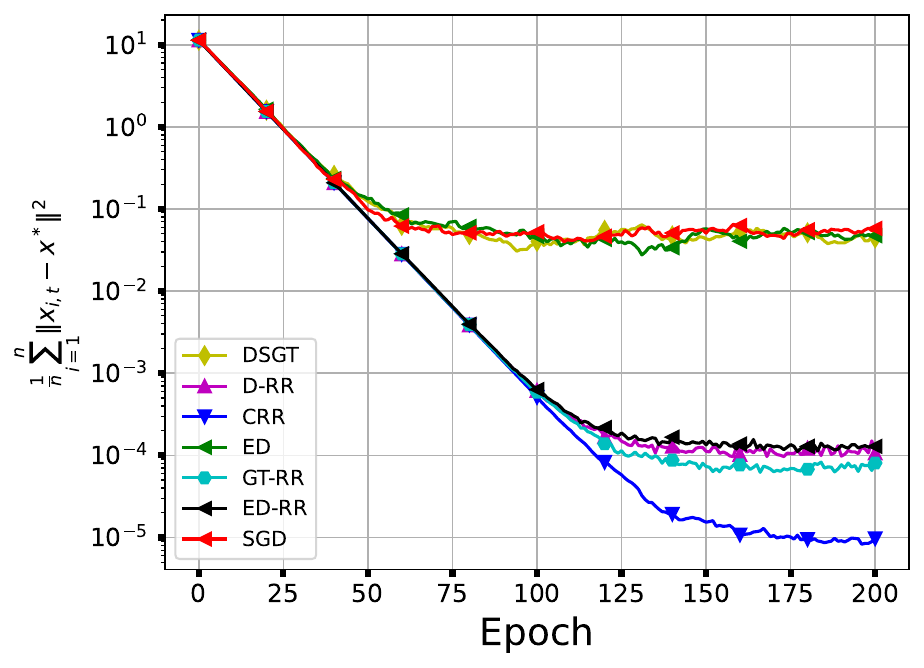}}
        \caption{Comparison among ED-RR, GT-RR, D-RR, ED, DSGT, SGD, and centralized RR for solving Problem \eqref{eq:ls}. The stepsizes are set as $1\times 10^{-5}$ for both graphs.}
        \label{fig:ls_c}
    \end{figure}

    \begin{figure}[htbp]
        \centering
        \subfloat[Ring graph, $n = 16$.]{\includegraphics[width = 0.245\textwidth]{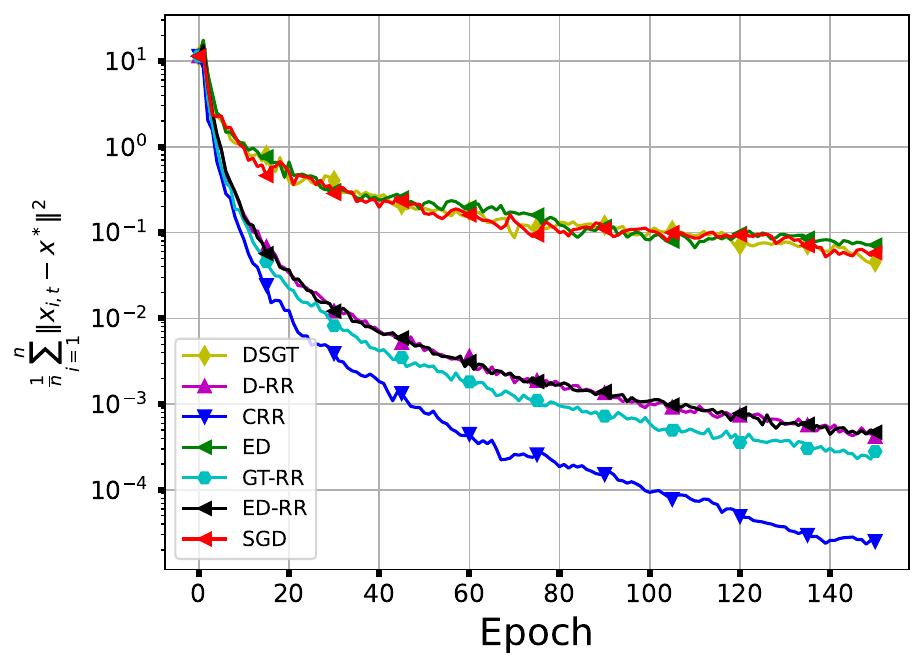}}
        \subfloat[Grid graph, $n = 16$.]{\includegraphics[width = 0.245\textwidth]{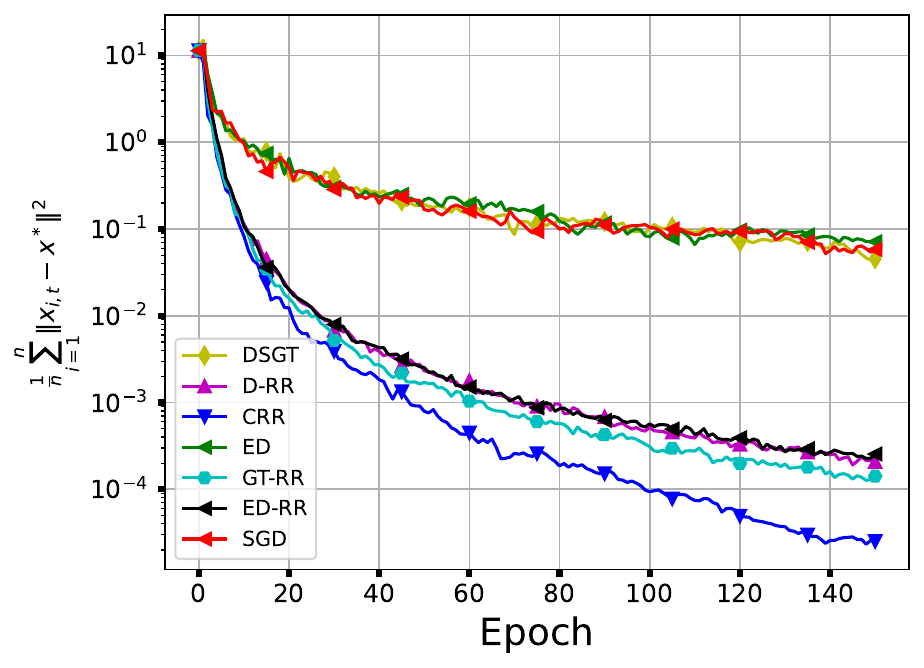}}
        \caption{Comparison among ED-RR, GT-RR, D-RR, ED, DSGT, SGD, and centralized RR for solving Problem \eqref{eq:ls}. The stepsizes are set as $1/(500t + 500)$ for both graphs.}
        \label{fig:ls_d}
    \end{figure}

    \subsection{Neural network}
    
    In this part, we compare the performance of the aforementioned methods in training a neural network with one hidden layer on the MNIST dataset \cite{mnist} over a ring graph and a grid graph, respectively. 
    Following the common approach for training machine learning models, we reduce the stepsizes on a plateau when comparing the performance of different methods, that is, we reduce the stepsizes if the measure metric does not decrease. The results are depicted in Fig. \ref{fig:nn}.
    We observe that GT-RR and ED-RR achieve better empirical performance compared to DSGT and ED, as well as D-RR. Surprisingly, GT-RR and ED-RR perform slightly better than CRR.

    \begin{figure}[htbp]
        \centering
        \subfloat[Ring graph, $n = 16$.]{\includegraphics[width = 0.245\textwidth]{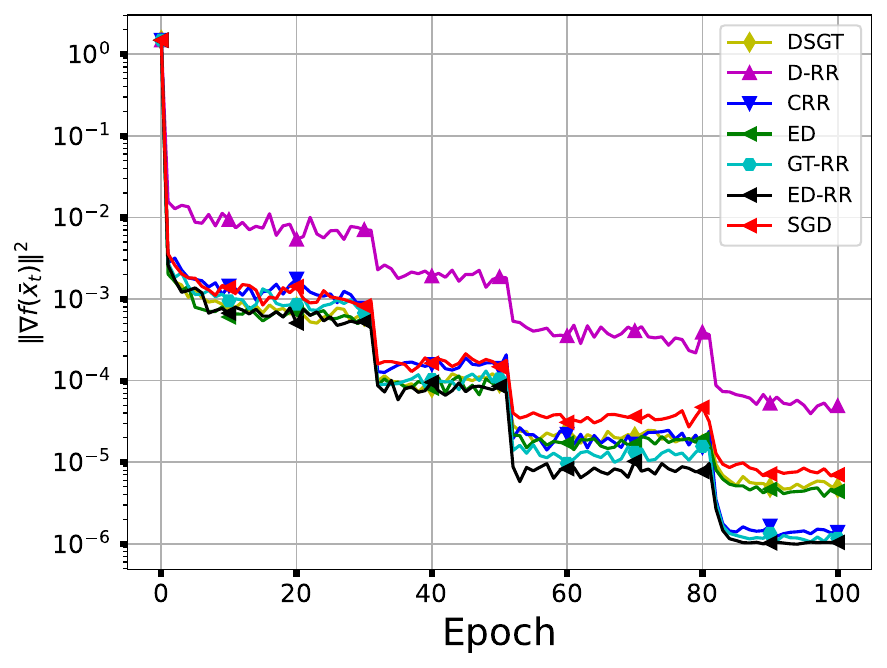}}
        \subfloat[Grid graph, $n = 16$.]{\includegraphics[width = 0.245\textwidth]{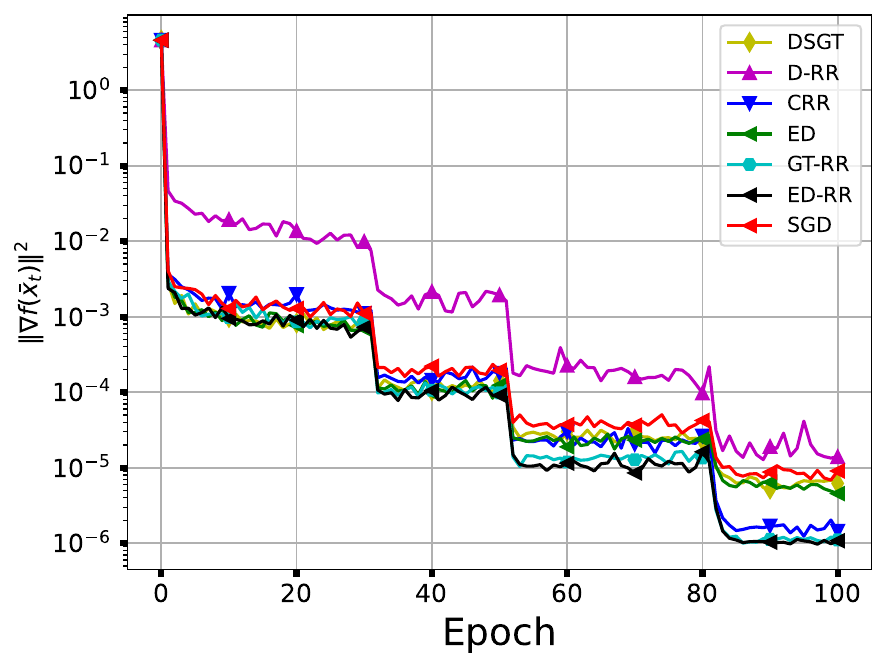}}
        \caption{Comparison among ED-RR, GT-RR, D-RR, ED, DSGT, SGD, and centralized RR for training a neural network on the MNIST dataset using constant stepsizes. The stepsizes are sequentially set as $1 / 2$, $1 / 10$, $1 / 50$, and $1 / 250$.}
        \label{fig:nn}
    \end{figure}

    \section{Conclusions}
    \label{sec:conclusions}

    This paper focuses on addressing the distributed optimization problem over networked agents. The proposed algorithms, Gradient Tracking with Random Reshuffling (GT-RR) and Exact Diffusion with Random Reshuffling (ED-RR), achieve theoretical guarantees comparable to the centralized RR method concerning the sample size $m$ and the epoch counter $T$ for smooth nonconvex objective functions and objective functions satisfying the PL condition. The obtained results notably improve upon previous works, particularly in terms of the sample size $m$ and the spectral gap $(1-\lambda)$ corresponding to the graph topology.
    In addition, experimental results corroborate the theoretical findings and demonstrate that the proposed algorithms, GT-RR and ED-RR, achieve better empirical performance under various network topologies, outperforming their unshuffled counterparts as well as D-RR. We believe that the proposed methods offer valuable insights into solving distributed optimization and machine learning problems over networked agents in real-world scenarios with improved efficiency and effectiveness.

\appendices

\section{Centralized RR}
\label{app:crr}
    \begin{algorithm}
		\caption{Centralized Random Reshuffling (C-RR)}
		\label{alg:CRR}
		\begin{algorithmic}[1]
			\STATE Initialize $x_{0}$ and stepsize $\alpha_t$.
			\FOR{Epoch $t = 0, 1, 2,\dots, T-1$}
			\STATE Sample $\crk{\pi_{0}, \pi_{1},\dots, \pi_{m-1}}$ of $[m]$
			\STATE Set $x_{t}^0 = x_{t}$
			\FOR{$\ell = 0, 1,\dots, m-1$}
			\STATE Update $x_{t}^{\ell + 1} = x_t^{\ell} -  {\frac{\alpha_t}{n}}\sumn \nabla f_{i, \pi_\ell}(x_t^{\ell})$.
			\ENDFOR
			\STATE Set $x_{t+ 1} = x_{t}^m$.
			\ENDFOR
			\STATE Output $x_{i, T}$.
		\end{algorithmic}
	\end{algorithm}

\section{Proofs}

\subsection{Proof of Lemma \ref{lem:abc}}
\label{app:lem_abc}
    Eliminating $\s_t^\ell$ in \eqref{eq:abc}, we obtain $\x_t^{\ell+1}-\x_t^\ell
    = (AC-B^2)\x_t^\ell-AC\x_t^{\ell-1}-\alpha_tA\prt{\nabla \Fp{\ell}(\x_t^\ell)-\nabla \Fp{\ell-1}(\x_t^{\ell-1})}$, $\ell \in [m-1]$.

    Then,
   \begin{equation}
    \label{eq:abc_x}
     \begin{aligned}
        &\x_t^1 = AC\x_t^0 - \alpha_t A\nabla \Fp{0}(\x_t^0)-B\z_t^0,\\
        &\x_t^{\ell + 1} = (AC - B^2 + I)\x_t^{\ell} - AC\x_t^{\ell - 1}\\
        &\quad -\alpha_tA\prt{\nabla \Fp{\ell}(\x_t^\ell)-\nabla \Fp{\ell-1}(\x_t^{\ell-1})},\ell\in [m-1]. 
     \end{aligned}
   \end{equation}

    \textbf{GT-RR}. We consider the compact form of GT-RR:
    \begin{equation}
        \label{eq:GT-RR}
        \begin{aligned}
            \x_{t}^{\ell + 1} & = W(\x_t^\ell - \alpha_t \y_t^\ell)\\
            \y_t^{\ell + 1} &= W\y_t^\ell + \nabla \Fp{\ell + 1}(\x_t^{\ell + 1}) - \nabla \Fp{\ell}(\x_t^{\ell}).
        \end{aligned}
    \end{equation}
    
    Eliminating $\y_t^\ell$ in \eqref{eq:GT-RR}, we have $\x_t^{\ell + 1} - W\x_t^\ell = W\x_t^\ell - W^2\x_t^{\ell -1} - \alpha_tW\brki{\nabla \Fp{\ell}(\x_t^\ell) - \nabla \Fp{\ell - 1}(\x_t^{\ell - 1})}.$

    Choosing $A = W$, $B = I-W$, and $C= W$ in \eqref{eq:abc_x} recovers the above update.

    \textbf{ED-RR}. Eliminating $\d_t^\ell$ in \eqref{eq:ED-RR} yields $\x_t^{\ell+1} - \x_t^{\ell}= W\x_t^\ell - W\x_t^{\ell - 1} - \alpha_t W\brki{\nabla \Fp{\ell}(\x_t^{\ell}) - \nabla \Fp{\ell - 1}(\x_t^{\ell - 1})} - (I-W)\x_t^\ell.$

    Letting $A = W$, $B = (I-W)^{1/2}$, and $C = I$ in \eqref{eq:abc_x} leads to the above relation. 
    Comparing the update forms of $\x_t^\ell$ in \eqref{eq:abc_epoch} and \eqref{eq:abc} leads to the definition of $\s_t^\ell$.

\subsection{Proof of Lemma \ref{lem:re}}
\label{app:re}

    The arguments are similar to those in \cite[Transformation II]{alghunaim2021unified}. We present them here for completeness.
    We have from the decomposition of $W$ in \eqref{eq:W_decomp} and Assumption \ref{as:abc} that
    \begin{align}
        A&=U\Lambda_a U^{\T}=
        \begin{pmatrix}
            \frac{1}{\sqrt{n}}\1 &\hU \\
        \end{pmatrix}
        \begin{pmatrix}
            1 &\0 \\
            \0 &\hat{\Lambda}_a \\
        \end{pmatrix}
        \begin{pmatrix}
            \frac{1}{\sqrt{n}}\1^{\T} \\
            \hU^{\T}
        \end{pmatrix}.
    \end{align}
    Similarly, we have $B^2 = U\diag(0, \hLambda_b^2)U^{\T}$ and $C=U\diag(1, \hLambda_c)U^{\T}$,
    where $\hat{\Lambda}_a=\diag(\lambda_{a,2}, \ldots, \lambda_{a,n})$, $\hat{\Lambda}_b^2=\diag(\lambda_{b,2}^2, \ldots, \lambda_{b,n}^2)$, and $\hat{\Lambda}_c=\diag(\lambda_{c,2}, \ldots, \lambda_{c,n})$.
    Here, $\lambda_{a,i}:=\sum_{d=0}^{p}a_d\lambda_{i}^{d}$, $\lambda_{b,i}^2:=\sum_{d=0}^{p}b_d\lambda_{i}^{d}$ and $\lambda_{c,i}:=\sum_{d=0}^{p}c_d\lambda_{i}^{d}$. Multiplying both sides of \eqref{eq:abc} by $U^{\T}$, it follows that $U^{\T}\x_t^{\ell + 1} = (\Lambda_a\Lambda_c - \Lambda_b^2)U^{\T}\x_t^\ell - \alpha_t \Lambda_a U^{\T}(\nabla \Fp{\ell}(\x_t^\ell) - \nabla F(\Bxt{0})) - U^{\T}\s_t^\ell$ and 
    $U^{\T}\s_t^{\ell + 1} = U^{\T}\s_t^\ell + \Lambda_b^2U^{\T}\x_t^\ell$.
   
   Hence, we obtain $\bar{x}_t^{\ell+1} = \bar{x}_t^{\ell}-\alpha_t \bar{\nabla} \Fp{\ell}(\x_t^\ell)$ and 
        \begin{equation}
            \label{eq:matrix_form}
            \begin{aligned}
                &\begin{pmatrix}
                    \hU^{\T}\x_t^{\ell + 1} \\
                    \hLambda_b^{-1}\hU^{\T}\s_t^{\ell + 1} \\
                \end{pmatrix}
                =
                \underbrace{\begin{pmatrix}
                    \hLambda_a\hLambda_c - \hLambda_b^2 & - \hLambda_b \\
                    \hLambda_b & I \\
                \end{pmatrix}}_{G}
                \begin{pmatrix}
                    \hU^{\T}\x_t^{\ell} \\
                    \hLambda_b^{-1}\hU^{\T}\s_t^{\ell} \\
                \end{pmatrix}\\ 
                &\quad -\alpha_t
                \begin{pmatrix}
                    \hLambda_a \hU^{\T}(\nabla \Fp{\ell}(\x_t^\ell) - \nabla F(\Bxt{0})) \\
                    \bzero \\
                \end{pmatrix}.
            \end{aligned}
        \end{equation}
    We then conclude from Lemma 1 in \cite{alghunaim2021unified} that $G$ admits the similarity decomposition $G = V\Gamma V^{-1}$ if the spectral norm of $G$ is strictly less than one. The desired result then follows by multiplying both sides of \eqref{eq:matrix_form} by $V^{-1}$.
    






\subsection{Proof of Lemma \ref{lem:cons_e}}
\label{app:cons_e}
    First, we have the decomposition 
    \begin{equation}
        \label{eq:decomp}
        \begin{aligned}
            &\frac{1}{n}\norm{\x_t^\ell - \Bxt{0}}^2 = \frac{1}{n}\norm{\x_t^\ell - \Bxt{\ell}}^2 + \norm{\bx{t}{\ell} - \bx{t}{0}}^2\\
            &\leq \frac{\norm{V}^2\norm{\e_t^\ell}^2}{n} + \norm{\bx{t}{\ell} - \bx{t}{0}}^2, \ \ell = 0,1,2,\ldots,m.
        \end{aligned}
    \end{equation}

    Denote $V^{-1}:= (V^{-1}_1, V^{-1}_2)$, where $V^{-1}_1, V^{-1}_2\in\R^{2(n-1)\times (n-1)}$ and $\normi{V_1^{-1}}\leq \normi{V^{-1}}$. Note that $V^{-1}_1$ (resp. $V^{-1}_2$) does not represent the inverse of $V_1$ (resp. $V_2$) but the left (resp. right) part of $V^{-1}$. Denote 
    $M_1^{\ell - j}:= (\gamma^2)^{\ell -j}\normi{V_1^{-1}}^2\normi{\hLambda_a}^2$. For any $\ell = 0,1,\ldots, m-1$, we have from Lemma \ref{lem:re_epoch} and the unbiasedness condition \eqref{eq:epoch_unb} that 
    \small
    \begin{equation}
        \label{eq:cond_e_s1}
        \begin{aligned}
            &\condE{\norm{\e_t^{\ell + 1}}^2}{\cF_t^0}
            \leq (1 + q)\prt{\gamma^2}^{\ell + 1}\norm{\e_t^0}^2 \\
            &+ \alpha^2_t(3+q^{-1})(\ell+1) \sum_{j=0}^{\ell}M_1^{\ell - j}L^2\condE{\norm{\x_t^j - \Bxt{0}}^2}{\cF_t^0}\\
            &+ 3\alpha_t^2(\ell + 1)\sum_{j=0}^\ell M_1^{\ell - j}\sumn \condE{\norm{\nabla \fp{i}{j}(\bx{t}{0}) - \nabla f_i(\bx{t}{0})}^2}{\cF_t^0},
        \end{aligned}
    \end{equation}\normalsize
    where we invoke $L$-smoothness of $f_{i,\ell}$ and Young's inequality for any $q>0$. 
    
    In light of \eqref{eq:decomp}, choosing $\ell = m-1$ and $q = (\frac{1+\gamma^2}{2\gamma^{2}})^{m-1}$ in \eqref{eq:cond_e_s1}, we obtain the contractive property of $\E[\normi{\e_t^0}^2|\cF_t^0]$ (w.r.t. itself) in two consecutive epochs, that is,
    \begin{equation}
        \label{eq:e_epoch}
        \begin{aligned}
            & \condE{\norm{\e_{t+1}^0}^2}{\cF_t^0} = \condE{\norm{\e_t^m}^2}{\cF_t^0} \leq 
            (\frac{1 + \gamma^2}{2})^m\norm{\e_t^0}^2\\
            &\quad + \frac{3\alpha_t^2m\C_4}{\normi{V}^2}\sum_{j=0}^{m-1}\sumn \condE{\norm{\nabla \fp{i}{j}(\bx{t}{0}) - \nabla f_i(\bx{t}{0})}^2}{\cF_t^0}\\
            &\quad + \frac{3\alpha_t^2mn\C_4L^2}{\normi{V}^2(1-\gamma^{2})}\condE{\cL_t}{\cF_t^0}.
        \end{aligned}
    \end{equation} 

    Finally, invoking Lemma \ref{lem:bounded_var} leads to the desired result.

\subsection{Proof of Lemma \ref{lem:cL}}
\label{app:cL}
    We first consider the term $\sum_{\ell=0}^{m-1}\condEi{\norm{\e_t^{\ell}}^2}{\cF_t^0}$. From \eqref{eq:cond_e_s1}, \eqref{eq:decomp}, and $\normi{V_1^{-1}}^2\leq \normi{V^{-1}}^2$, by choosing $q = (\frac{1+\gamma^2}{2\gamma^{2}})^{\ell+1}-1$, we have for $\ell = 0,1,\ldots, m-1$ that
    \small
    \begin{equation}
        \label{eq:e_iteration}
        \begin{aligned}
            &\condE{\norm{\e_t^{\ell+1}}^2}{\cF_t^0} \leq 
            (\frac{1 + \gamma^2}{2})^{\ell+1}\norm{\e_t^0}^2\\
            & + \frac{3\alpha_t^2mn\C_4L^2}{\normi{V}^2(1-\gamma^2)}\sum_{j=0}^\ell\crk{\frac{\normi{V}^2}{n}\condE{\norm{\e_t^j}^2 + \norm{\bx{t}{j} - \bx{t}{0}}^2}{\cF_t^0}}\\
            &+ \frac{3\alpha_t^2m\C_4}{\normi{V}^2}\sum_{j=0}^{\ell}(\gamma^2)^{\ell - j}\sumn \condE{\norm{\nabla \fp{i}{j}(\bx{t}{0}) - \nabla f_i(\bx{t}{0})}^2}{\cF_t^0}.
        \end{aligned}
    \end{equation} \normalsize

    Summing up from $\ell=0$ to $m-1$ on both sides of the above inequality, it follows that
    \small
    \begin{equation}
        \label{eq:cL_cons}
        \begin{aligned}
            &\sum_{\ell=0}^{m-1}\condE{\norm{\e_t^{\ell}}^2}{\cF_t^0} 
            \leq
            \frac{1-(\frac{1+\gamma^2}{2})^m}{1-\frac{1+\gamma^2}{2}}\norm{\e_t^0}^2\\
            & + \frac{3\alpha_t^2m^2nL^2\C_4}{\normi{V}^2(1-\gamma^{2})}\condE{\cL_t}{\cF_t^0}+ \frac{6\alpha_t^2m^2nL\C_4}{\normi{V}^2(1-\gamma^2)}\brk{f(\bx{t}{0}) - f^* +\sigfmn},\\
        \end{aligned}
    \end{equation}   \normalsize
    where we invoke Lemma \ref{lem:bounded_var}. 

    We next consider the term $\sum_{\ell = 0}^{m-1}\E[\norm{\bx{t}{\ell} - \bx{t}{0}}^2|\cF_t^0]$ in light of Lemma \ref{lem:rr}. Recall the relation that $\bx{t}{\ell} - \bx{t}{0}= \alpha_t/n\sum_{j=0}^{\ell - 1}\sumn\nabla\fp{i}{j}(x_{i,t}^j)$. Then,
    \small
    \begin{equation}
        \label{eq:cL_bar0}
        \begin{aligned}
            &\condE{\norm{\bx{t}{\ell} - \bx{t}{0}}^2}{\cF_t^0}
            \leq \frac{2\alpha_t^2 L^2\ell}{n}\sumn\sum_{j=0}^{\ell - 1}\condE{\norm{x_{i,t}^{j} - \bx{t}{0}}^2}{\cF_t^0}\\
            & + 2\alpha_t^2\condE{\norm{\frac{1}{n}\sum_{j=0}^{\ell - 1}\sumn \nabla \fp{i}{j}(\bx{t}{0})}^2}{\cF_t^0}.
        \end{aligned}
    \end{equation}\normalsize

    We then consider the last term in \eqref{eq:cL_bar0}. Compared to the estimate in \cite[supplementary material (57)]{huang2023drr}, the upper bound in \eqref{eq:cL_bar1} improves the order of $m$ due to Lemma \ref{lem:rr}.
    \small
    \begin{equation}
        \label{eq:cL_bar1}
        \begin{aligned}
            &\frac{1}{2}\condE{\norm{\frac{1}{n}\sum_{j=0}^{\ell - 1}\sumn \nabla \fp{i}{j}(\bx{t}{0})}^2}{\cF_t^0} - \ell^2\norm{\nabla f(\bx{t}{0})}^2\\
            &\leq \ell^2\condE{\norm{\frac{1}{n\ell}\sum_{j=0}^{\ell - 1}\sumn\brk{\nabla \fp{i}{j}(\bx{t}{0}) -\nabla f_i(\bx{t}{0})}}^2}{\cF_t^0}\\
            &\leq \frac{\ell^2}{n}\sumn \condE{\norm{\frac{1}{\ell}\sum_{j=0}^{\ell - 1}\brk{\nabla \fp{i}{j}(\bx{t}{0}) - \nabla f_i(\bx{t}{0})}}^2}{\cF_t^0}\\
            &= \frac{\ell^2}{n} \frac{m-\ell}{\ell(m - 1)}\sumn\crk{\frac{1}{m}\sum_{\ell = 1}^m\condE{\norm{\nabla f_{i,\ell}(\bx{t}{0}) - \nabla f_i(\bx{t}{0})}^2}{\cF_t^0}},
        \end{aligned}
    \end{equation}\normalsize
    where the last equality holds due to Lemma \ref{lem:rr} by setting $\bar{X}_\pi = \frac{1}{\ell}\sum_{j = 0}^{\ell - 1}\nabla \fp{i}{j}(\bx{t}{0})$,  $\bar{X} = \nabla f_i(\bx{t}{0})= \frac{1}{m}\sum_{\ell = 1}^{m} \nabla f_{i,\ell}(\bx{t}{0})$, and $\sigma^2= \frac{1}{m}\sum_{\ell = 1}^m\condEi{\normi{\nabla f_{i,\ell}(\bx{t}{0}) - \nabla f_i(\bx{t}{0})}^2}{\cF_t^0}.$

    Therefore, we obtain the improved upper bound for $\sum_{\ell = 0}^{m-1}\condEi{\normi{\bx{t}{\ell} - \bx{t}{0}}^2}{\cF_t^0}$ in \eqref{eq:cL_bar}:
    \small
    \begin{equation}
        \label{eq:cL_bar}
        \begin{aligned}
            &\sum_{\ell = 0}^{m-1}\condE{\norm{\bx{t}{\ell} - \bx{t}{0}}^2}{\cF_t^0} 
            \leq \alpha_t^2m^2 L^2\condE{\cL_t}{\cF_t^0}\\
            & + 2\alpha_t^2m^3\norm{\nabla f(\bx{t}{0})}^2+ \frac{4\alpha_t^2 Lm(m+1)}{3}\brk{f(\bx{t}{0}) - f^* + \sigfmn}.
        \end{aligned}
    \end{equation}\normalsize


    Combining \eqref{eq:cL_cons} and \eqref{eq:cL_bar}, letting $\alpha_t\leq \min\crki{1/(2mL), (1-\gamma^2)/(4mL\sqrt{\C_4})}$ yield the desired result.

\subsection{Proof of Lemma \ref{lem:final_re}}
\label{app:final_re}
    Inspired by \cite[Lemma 15]{huang2023drr}, we define $Q_t$ as $Q_t := f(\bx{t}{0}) - f^* + \alpha_t \C_0 \normi{\e_t^0}^2,$
    where $\C_0>0$ is to be determined later. Substituting the result of Lemma \ref{lem:cL} into Lemmas \ref{lem:cons_e} and \ref{lem:descent}, we obtain for $\alpha_t \leq (1-\gamma^2)/\sqrt{4mL^2\C_1}$ that 
    \small
    \begin{equation}
        \label{eq:et_cL}
        \begin{aligned}
            &\condE{\norm{\e_{t + 1}^0}^2}{\cF_t^0} 
            \leq \brk{\prt{\frac{1+\gamma^2}{2}}^m + \frac{12\alpha_t^2 m L^2 \C_2}{1-\gamma^2}}\norm{\e_t^0}^2\\
            &\quad + 12\alpha_t^2m^2nL\norm{V^{-1}}^2\norm{\hLambda_a}^2\brk{f(\bx{t}{0}) - f^* + \sigfmn}\\
            &\quad + \frac{12\alpha_t^4m^4n L^2\norm{V^{-1}}^2\norm{\hLambda_a}^2}{1-\gamma^2}\norm{\nabla f(\bx{t}{0})}^2,
        \end{aligned}
    \end{equation}\normalsize
    and 
    \begin{equation}
        \label{eq:f_cL}
        \begin{aligned}
            &\condE{f(\bx{t + 1}{0}) - f^*}{\cF_t^0} \leq \prt{1 + \frac{8\alpha_t^3m^2 L^3\C_1}{1-\gamma^2}}\brk{f(\bx{t}{0}) - f^*}\\
            &\quad + \frac{4\alpha_t L^2\C_2}{n\normi{V^{-1}}^2\normi{\hLambda_a}^2}\norm{\e_t^0}^2 + \frac{8\alpha_t^3m^2L^3\C_1}{1-\gamma^2}\sigfmn\\
            &\quad- \frac{\alpha_t m }{2}\prt{1 - 8\alpha_t^2m^2 L^2}\norm{\nabla f(\bx{t}{0})}^2.
        \end{aligned}
    \end{equation}

    Therefore, 
    \begin{equation}
        \label{eq:Qt_can}
        \begin{aligned}
            &\condE{Q_{t + 1}}{\cF_t^0} \leq M_2\brk{f(\bx{t}{0}) - f^*}+ M_3\norm{\e_t^0}^2\\
            &- \frac{\alpha_t m M_4}{2}\norm{\nabla f(\bx{t}{0})}^2+ M_5\sigfmn,
        \end{aligned}
    \end{equation}
    where
    \small
    \begin{align*}
        M_2&:=  1 + \frac{8\alpha_t^3m^2 L^3\C_1}{1-\gamma^2} +12\alpha_t^3m^2nL\norm{V^{-1}}^2\norm{\hLambda_a}^2\C_0,\\
        M_3&:= \brk{\prt{\frac{1+\gamma^2}{2}}^m + \frac{12\alpha_t^2 m L^2 \C_2}{1-\gamma^2}}\alpha_t\C_0 + \frac{4\alpha_t L^2\C_2}{n\normi{V^{-1}}^2\normi{\hLambda_a}^2},\\
        M_4&:= 1 - 8\alpha_t^2m^2 L^2 - \frac{24\alpha_t^4m^3n L^2\normi{V^{-1}}^2\normi{\hLambda_a}^2\C_0}{1-\gamma^2},\\
        M_5&:= 12\alpha_t^3m^2nL\norm{V^{-1}}^2\norm{\hLambda_a}^2\C_0 + \frac{8\alpha_t^3m^2L^3\C_1}{1-\gamma^2}.
    \end{align*}\normalsize

    Letting $\alpha_t \leq (1-\gamma^2)/(2L\sqrt{6m\C_4})$,
    it suffices to choose $\C_0:= 8L^2\normi{V}^2/[n(1-\gamma^2)]$
    such that $M_3 \leq \alpha_t \C_0$.

    Finally, letting $\C_3:= 12\C_4 + \C_1$ and $\alpha_t\leq \min\crki{1/(4\sqrt{2}mL), \sqrt{1-\gamma^2}/\prti{6m^{3/4}\C_4^{1/4} L}}$
    yields the desired result.



\subsection{Proof of Theorem \ref{thm:nonconvex_final_results}}
\label{app:nonconvex_final_results}

        We first state Lemma \ref{lem:hhelp} from  \cite[Lemma 6]{mishchenko2020random} that provides a direct link connecting Lemma \ref{lem:final_re} to the desired results.

        \begin{lemma}
        \label{lem:hhelp}
        If there exists constants $a,b,c \geq 0$ and non-negative sequences ${s_t},{q_t}$, for any $t$, we have $ s_{t+1} \leq (1+a)s_{t}-bq_t+c,$
        then it holds that $\min\limits_{t=0,1,...,T-1} q_t \leq \frac{(1+a)^T}{bT}s_0+\frac{c}{b}.$
        \end{lemma}

    Taking full expectation on the result given in Lemma \ref{lem:final_re} and applying Lemma \ref{lem:hhelp}, we obtain
    \small
    \begin{equation}
        \label{eq:minE}
        \begin{aligned}
            &\min\limits_{t=0,\ldots,T-1}\E\brk{\norm{{\nabla}f(\bar{x}_t^{0})}^2}
            \leq \frac{12 Q_0}{\alpha m T} + \frac{32\alpha^2m L^3\C_3}{1-\gamma^2}\sigfmn,
        \end{aligned}
    \end{equation}\normalsize
    where the stepsize $\alpha$ is set to satisfy $\alpha \leq \min\crki{\brki{\prti{1-\gamma^2}/\prti{16m^2 L^3\C_1 T}}^{1/3}, \brki{\prti{1-\gamma^2}/\prti{192 m^2 L^3\C_4 T}}^{1/3}}.$

    If we set $\beta \geq 2\sqrt{2}(1-\gamma^2) + 3\prt{\frac{(1-\gamma^2)^2\C_1}{m}}^{\frac{1}{4}} + \sqrt{\frac{\C_1}{m}} + \sqrt{\frac{6\C_4}{m}}$ and $\alpha = \frac{1}{\frac{2mL\beta}{1-\gamma^2} + \prt{\frac{192 m^2 L^3 \C_1 T}{1-\gamma^2}}^{\frac{1}{3}}}$,
    then the stepsize $\alpha$ satisfies condition \eqref{eq:alpha_ncvx}, and 
    \small
    \begin{align}
        \label{eq:alpha}
        \alpha^2 \leq \prt{\frac{1-\gamma^2}{192 m^2 L^3 \C_1 T}}^{\frac{2}{3}},\ \frac{1}{\alpha mT} = \frac{ \frac{2mL\beta}{1-\gamma^2} + \prt{\frac{192 m^2 L^3 \C_1 T}{1-\gamma^2}}^{\frac{1}{3}}}{mT}.
    \end{align}\normalsize

    Substituting \eqref{eq:alpha} into \eqref{eq:minE} yields the desired result.


\subsection{Proof of Corollary \ref{cor:nonconvex_DGT_RR}}
\label{app:nonconvex_DGT_RR}
        According to \cite[Appendix B2]{alghunaim2021unified}, we have $\normi{V}^2\leq 3$, $\normi{V_1^{-1}}^2\leq\normi{V^{-1}}^2\leq 9$, $\normi{\hLambda_a}^2 = \lambda^2$, $\normi{\hLambda_b^{-1}}^2 =1 / (1-\lambda)^2$, and $\gamma= \lambda.$
        Then, 
        \small
        \begin{equation}
            \label{eq:Cs_gt}
            \begin{aligned}
                &\frac{1}{3}\leq C_1\leq 42,\ \C_2\leq \frac{27}{1-\lambda^2},\ \C_3\leq 366,\ \C_4\leq 27, \\
                &\beta_1 = 2\sqrt{2}(1-\lambda^2) + 3\prt{\frac{42(1-\lambda^2)^2}{m}}^{\frac{1}{4}} + \sqrt{\frac{42}{m}}+ \sqrt{\frac{162}{m}}.
            \end{aligned}
        \end{equation}\normalsize
            
        Plugging \eqref{eq:Cs_gt} into Theorem \ref{thm:nonconvex_final_results} yields the desired result.

        

\subsection{Proof of Corollary \ref{cor:nonconvex_ED_RR}}
\label{app:nonconvex_ED_RR}
        Similar to \cite[Appendix B.2]{alghunaim2021unified}, we have $\normi{V}^2\leq 4$, $\normi{V_1^{-1}}^2\leq\normi{V^{-1}}^2\leq 2/ \underline{\lambda}$, $\normi{\hLambda_a}^2 = \lambda^2$, $\normi{\hLambda_b^{-1}}^2=1/(1-\lambda)$, and  $\gamma= \sqrt{\lambda}$,
        where $\underline{\lambda}=\min\{\lambda_i|\lambda_i>0\}_{i=2}^n$. 
        Then, $1/3\leq \C_1 \leq 38/\prti{3\underline{\lambda}}$, $\C_2\leq 8/\prti{\underline{\lambda}(1-\lambda)}$, $\C_3\leq 326/\prti{3\underline{\lambda}}$, $\C_4\leq 8/\prti{\underline{\lambda}}$, and $\beta_2 =  2\sqrt{2}(1-\lambda) + 3\prti{\frac{38(1-\lambda)^2}{3\underline{\lambda}m}}^{1/4} + \sqrt{\frac{38}{3\underline{\lambda}m}} + \sqrt{\frac{48}{\underline{\lambda}m}}.$
        Invoking Theorem \ref{thm:nonconvex_final_results} yields the desired result.
        

    \subsection{Proof of Theorem \ref{thm:drr}}
    \label{app:thm_drr}
    
        The improvements on the convergence rates of D-RR relate to the term $\sum_{\ell=0}^{m-1}\E\brki{\normi{\bx{t}{\ell} - \bx{t}{0}}^2|\cF_t^0}$ and the construction of the Lyapunov function. Since relation $\bx{t}{\ell}-\bx{t}{0} = \alpha_t^2 / n\sum_{j=0}^{\ell-1} \sumn \nabla \fp{i}{j}(x_{i,t}^j)$ holds for D-RR, GT-RR, as well as ED-RR, we have the same recursion regarding  $\sum_{\ell=0}^{m-1}\E\brki{\normi{\bx{t}{\ell} - \bx{t}{0}}^2|\cF_t^0}$ for D-RR as in \eqref{eq:cL_bar}.
        
        Analogous to the definition of $\cL_t$ in \eqref{eq:cLt}, we define the auxiliary term $\Ldrr_t$ for D-RR as
        \small
        \begin{equation}
            \label{eq:drr_cL}
            \begin{aligned}
                \Ldrr_t&:= \frac{1}{n}\sum_{\ell=0}^{m-1}\norm{\x_t^\ell - \Bxt{\ell}}^2 + \sum_{\ell=0}^{m-1}\norm{\bx{t}{\ell} - \bx{t}{0}}^2
            \end{aligned}
        \end{equation}\normalsize

        Similar to \cite[supplementary material (53)-(56)]{huang2023drr}, we have 
        \small
        \begin{equation}
            \label{eq:drr_sum}
            \begin{aligned}
                &\sum_{\ell=0}^{m-1}\norm{\x_t^{\ell} - \1(\bx{t}{\ell})^{\T}}^2\leq \frac{2\brk{1 - \prt{\frac{1 + \lambda^2}{2}}^m}}{1-\lambda^2}\norm{\x_t^0 - \1(\bx{t}{0})^{\T}}^2\\
                &\quad + \frac{24L\alpha^2 m^2n \sigfmn}{(1-\lambda^2)^2}+ \frac{12mn\alpha^2}{(1-\lambda^2)^2}\norm{\nabla f(\bx{t}{0})}^2\\
                &\quad + \frac{24\alpha^2nL^2}{(1-\lambda^2)^2}\Ldrr_t + \frac{24 L \alpha^2 m^2n}{(1-\lambda^2)^2}\prt{f(\bx{t}{0}) - f^*}.
            \end{aligned}
        \end{equation}\normalsize

        Combing \eqref{eq:drr_sum} and \eqref{eq:cL_bar} 
        and letting $\alpha\leq \min\crki{(1-\lambda^2)/(4\sqrt{6}L), 1/(2mL)}$,
        we have 
        \small
        \begin{equation}
            \label{eq:cL_drr}
            \begin{aligned}
                &\condE{\Ldrr_t}{\cF_t^0}\leq \frac{4\brk{1 - \prt{\frac{1 + \lambda^2}{2}}^m}}{n(1-\lambda^2)}\norm{\x_t^0 - \1(\bx{t}{0})^{\T}}^2\\
                &+ \frac{28\alpha^2 m^3}{(1-\lambda^2)^2}\norm{\nabla f(\bx{t}{0})}^2+ \frac{54\alpha^2 m^2 L}{(1-\lambda^2)^2}\brk{(f(\bx{t}{0}) - f^*) + \sigfmn}.
            \end{aligned}
        \end{equation}\normalsize

        It is worth noting that considering the epoch-wise error of $\normi{\x_t^0 - \Bxt{0}}^2$ according to Remark \ref{rem:cons_drr} would not improve the order of $(1-\lambda)$ compared to \cite[supplementary material (58)]{huang2023drr} of D-RR. We thus follow the original recursion in \cite[Lemma 14]{huang2023drr}, invoke \eqref{eq:cL_drr}, and let $\alpha\leq \min\crki{(1-\lambda^2)/(2mL\sqrt{14}), (1-\lambda^2)/(3L\sqrt{6m})}$: 
        \small
        \begin{equation}
            \label{eq:epoch_drr}
            \begin{aligned}
                &\condE{\norm{\x_{t+1}^0 - \1(\bx{t+1}{0})^{\T}}^2}{\cF_t^0}\\ 
                &\leq \crk{\prt{\frac{1 + \lambda^2}{2}}^m +  \frac{48\alpha^2 L^2\brk{1 - \prt{\frac{1 + \lambda^2}{2}}^m}}{(1-\lambda^2)^2}}\norm{\x_t^0 - \1(\bx{t}{0})^{\T}}^2\\
                &\quad + \frac{12\alpha^2mn}{1-\lambda^2}\norm{\nabla f(\bx{t}{0})}^2+ \frac{24\alpha^2 L mn}{1-\lambda^2}\brk{\prt{f(\bx{t}{0}) - f^*} + \sigfmn},
            \end{aligned}
        \end{equation}\normalsize

        
        Denote $\Qdrr_t := f(\bx{t}{0}) - f^* + \frac{8\alpha L^2}{n(1-\lambda^2)}\normi{\x_t^0 - \Bxt{0}}^2.$
        Letting the stepsize $\alpha$ satisfy $\alpha\leq \min \crki{\prti{1-\lambda^2}/\prti{4\sqrt{6} L}, \prti{1-\lambda^2}/\prti{4mL\sqrt{14}}, \prti{1-\lambda^2}/\prti{8L\sqrt{3}}},$
        we have from Lemma \ref{lem:descent} and \eqref{eq:epoch_drr} that 
        \small
        \begin{align*}
            &\condE{\Qdrr_{t + 1}}{\cF_t^0} 
            \leq \brk{1 + \frac{54\alpha^3 m^2L^3}{(1-\lambda^2)^2} + \frac{192\alpha^3 L^3 m}{(1-\lambda^2)^2}}\Qdrr_t\\
            &\quad - \frac{\alpha m}{4}\norm{\nabla f(\bx{t}{0})}^2 + \frac{246\alpha^3 m^2 L^3}{(1-\lambda^2)^2} \sigfmn.
        \end{align*}\normalsize

        Taking full expectation and applying Lemma \ref{lem:hhelp} yield the desired result.




\subsection{Proof of Lemma \ref{lem:cons_e_m}}
\label{app:cons_e_m}
        Denote $\C(\gamma, m):=\sum_{j=0}^{m-1}(\gamma^2)^{m-1-j}$. Setting $q =[(1+\gamma^2)/(2\gamma^{2})]^m-1$ and $\ell = m-1$ in \eqref{eq:cond_e_s1} yields 
        \small
        \begin{equation}
            \label{eq:cons_e_m_s1}
            \begin{aligned}
                &\condE{\norm{\e_t^{m}}^2}{\cF_t^0} 
                \leq \prt{\frac{1 + \gamma^2}{2}}^m\norm{\e_t^0}^2+ \frac{3\alpha^2_tmn L^2\C_4}{\normi{V}^2(1-\gamma^2)} \condE{\cL_t}{\cF_t^0}\\
                &+ \frac{3\alpha_t^2\C_4}{\normi{V}^2}\C(\gamma, m) \sumn\summ \condE{\norm{\nabla f_{i,\ell}(\bx{t}{0}) - \nabla f_i(\bx{t}{0})}^2}{\cF_t^0},
            \end{aligned}
        \end{equation}\normalsize
        where we apply the result of Lemma \ref{lem:rr} that $\condEi{\normi{\nabla \fp{i}{j}(\bx{t}{0}) - \nabla f_i(\bx{t}{0})}^2}{\cF_t^0}= \frac{1}{m}\summ \condEi{\normi{\nabla f_{i,\ell}(\bx{t}{0}) - \nabla f_i(\bx{t}{0})}^2}{\cF_t^0},\forall j=0,1,\ldots, m-1$.

        For inequality \eqref{eq:cons_e_m_s1}, invoking Lemma \ref{lem:bounded_var}, setting the stepsize $\alpha_t$ satisfy $\alpha_t \leq \min\crki{\prti{1-\gamma^2}/\prti{\sqrt{24mL^2 \C_4}}, \prti{1-\gamma^2}/\prti{\sqrt{4m^2L^2\C_1}}},$
        and substituting the result of Lemma \ref{lem:cL} yield the desired result.

\subsection{Proof of Theorem \ref{thm:linear}}
\label{app:linear}

    Our goal is to establish the recursion for $\frac{1}{n}\sumn\E\brki{f(x_{i,t}^{0})-f^*}$, whose relation to $\E\brki{f(\bx{t}{0}) - f^*}$ is given by
    \small
    \begin{align}
        \label{eq:point_error}
        \frac{1}{n}\sumn\prt{f(x_{i,t}^{0})-f^*}
        \leq 2\prt{f(\bar{x}_{t}^{0})-f^*} + \frac{L\normi{V}^2\normi{\e_t^0}^2}{n}.
    \end{align}\normalsize


    \textbf{Step 1: Derive the improved recursion of $\condE{Q_t}{\cF_t^0}$ under the PL condition.} 
    Invoking the PL condition (Assumption \ref{as:PL}) on Lemma \ref{lem:final_re} yields $\condEi{Q_{t + 1}}{\cF_t^0}
    \leq \prti{1 - \alpha_t\mu m/4} Q_t + \prti{8\alpha_t^3m^2L^3\C_3\sigfmn}/\prti{1-\gamma^2},$
    where we set $\alpha_t$ as
    \begin{align}
        \label{eq:a1}
        \alpha_t \leq \min\crk{\sqrt{\frac{(1-\gamma^2)\mu}{64 m L^3 \C_1}}, \sqrt{\frac{\mu(1-\gamma^2)}{768 mL^3 \C_4}}, {\frac{1-\gamma^2}{m\mu}}}.
    \end{align}

    \textbf{Option I: Constant stepsize}. If a constant stepsize $\alpha_t = \alpha$ is used, then, 
    \begin{equation}
        \label{eq:Qt_cstep}
        \begin{aligned}
            \E\brk{Q_t} &\leq \prt{1 - \frac{\alpha\mu m}{4}}^{t}Q_0 + \frac{32\alpha^2 m L^3\C_3}{\mu(1-\gamma^2)}\sigfmn.
        \end{aligned}
    \end{equation}

    \textbf{Option II: Decreasing stepsize.} Denote $\kappa:= L/\mu$. If decreasing stepsize $\alpha_t = \theta/(m\mu(t + K))$ is used for some constants $\theta, K >0$, then for $\theta > 16$ \cite[Lemma 11]{huang2023cedas} yields
    \begin{equation}
        \label{eq:Qt_dstep}
        \begin{aligned}
            \E\brk{Q_t} &\leq \prt{\frac{K}{t + K}}^{\frac{\theta}{4}} Q_0 + \frac{64\kappa^3\theta^3\C_3}{m(\theta - 8)(1-\gamma^2)(t + K)^2}\sigfmn.
        \end{aligned}
    \end{equation}

    \textbf{Option III: Uniform bound}. It can be verified from \eqref{eq:Qt_cstep} and \eqref{eq:Qt_dstep} that $\E[Q_t]$ has a uniform bound for both constant and decreasing stepsizes satisfying \eqref{eq:a1}, $\E\brk{Q_t} \leq Q_0 + 6\sigfmn.$

    \textbf{Step 2: Derive the decoupled recursion for $\E\brki{\normi{\e_t^0}^2}$.} We have $(f(\bx{t}{0}) - f^*)\leq Q_t$ for any $t\geq 0$. 
    Then, taking full expectation on the recursion of $\E[\normi{\e_t^0}^2|\cF_t^0]$ in Lemma \ref{lem:cons_e_m} yields
    \begin{equation}
        \label{eq:Qt_et}
        \begin{aligned}
            &\E\brk{\norm{\e_{t+1}^{0}}^2} 
            \leq q_0\E\brk{\norm{\e_t^0}^2}\\
            &+ 6\alpha_t^2mnL\norm{V^{-1}}^2\norm{\hLambda_a}^2\C_5\prt{\C(\gamma, m) + 1 + \frac{4\alpha_t^2m^3L^2}{1-\gamma^2}},
        \end{aligned}
    \end{equation}
    where $\C_5:=Q_0 + 7\sigfmn$.

    \textbf{Option I: Constant stepsize}. Consider a constant stepsize $\alpha_t = \alpha$. 
    Note that $\sum_{j=0}^{t-1}q_0^{t-1-j}\C(\gamma, m)
    \leq 4/\prti{1-\gamma^2}.$
    
    
    Then, \eqref{eq:Qt_et} becomes
    \begin{equation}
        \label{eq:et_cstep}
        \begin{aligned}
            &\E\brk{\norm{\e_{t}^0}^2}
            \leq q_0^{t} \norm{\e_0^0}^2 + \frac{48\alpha^2mnL\normi{V^{-1}}^2\norm{\hLambda_a}^2\C_5}{1-\gamma^2}\\
            &\quad + \frac{96\alpha^4m^4nL^3\normi{V^{-1}}^2\norm{\hLambda_a}^2\C_5}{(1-\gamma^2)^2}.
        \end{aligned}
    \end{equation}
    
    \textbf{Option II: Decreasing stepsize}. We first state Lemma \ref{lem:cons_unrolling} (\cite[Lemma 12]{huang2023cedas}) that helps us unrolling the inequality \eqref{eq:Qt_et} when we use decreasing stepsize $\alpha_t = \theta/(\mu m (t + K))$. 
    \begin{lemma}
        \label{lem:cons_unrolling}
        Suppose we have two sequences of positive numbers $\crki{a_k}$ and $\crki{b_k}$ satisfying $m\geq 1/\brki{1-(1 - b_1/2)^{1/p}}$ and $a_{k+1}\leq (1 - b_1)a_k + \sum_{i=2}^p \frac{b_i}{(k + m)^i}, b_1<1,$
        then $a_k \leq (1 - b_1)^k a_0 + \sum_{i=2}^p \frac{2b_i}{b_1(k + m)^i}.$ 
    \end{lemma}

    Therefore, the inequality \eqref{eq:Qt_et} becomes 
    \small
    \begin{equation}
        \label{eq:et_decreasing} 
        \begin{aligned}
            &\E\brk{\norm{\e_{t+1}^0}^2}
            \leq q_0^t \norm{\e_0^0}^2 +  \frac{48\theta^4 nL^3\normi{V^{-1}}^2\normi{\hLambda_a}^2\C_5 }{(1-q_0)\mu^4 (1-\gamma^2)(t+K)^4}\\
            & \frac{12\theta^2 nL\normi{V^{-1}}^2\normi{\hLambda_a}^2\C_5\prt{\C(\gamma, m) + 1}}{(1-q_0)m\mu^2 (t+K)^2},
        \end{aligned}
    \end{equation}\normalsize
    where we invoke Lemma \ref{lem:cons_unrolling} for $K\geq 32/(1-\gamma^2)$. 
    
    Noting that $1/(1-q_0)\leq 3/\prti{1-\gamma^2}$ and $\C(\gamma, m)/\prti{1-q_0} \leq 4/\prti{1-\gamma^2}$,
    the inequality \eqref{eq:et_decreasing} further leads to
    \begin{align*}
        &\E\brk{\norm{\e_{t}^0}^2} \leq q_0^t \norm{\e_0^0}^2 + \frac{144\theta^2 nL\normi{V^{-1}}^2\normi{\hLambda_a}^2\C_5}{(1-\gamma^2)m\mu^2 (t+K)^2}\\
        &\quad + \frac{96\theta^4 nL^3\normi{V^{-1}}^2\normi{\hLambda_a}^2\C_5 }{\mu^4 (1-\gamma^2)^2(t+K)^4}.
    \end{align*}

    \textbf{Step 3: Derive the total error bounds for both constant and decreasing stepsizes.} Applying relation \eqref{eq:point_error} to the corresponding options in Step 1 and 2 yields the desired results.


\bibliographystyle{IEEEtran}
\bibliography{references_all}

\end{document}